\documentclass[11pt]{article}
\usepackage{amsfonts}
\usepackage{amsmath}
\usepackage{enumerate}
\usepackage[parfill]{parskip}
\usepackage{color}

\newtheorem{theorem}{Theorem}

\newtheorem{assumption}{Assumption}

\newtheorem{corollary}{Corollary}

\newtheorem{definition}{Definition}
\newtheorem{example}{Example}

\newtheorem{lemma}{Lemma}

\newenvironment{proof}[1][Proof]{\textbf{#1.} }{\ \rule{0.5em}{0.5em}}
\begin{document}

\title{On progressive filtration expansion with a process}
\author{Younes Kchia\thanks{%
Centre de Math\'ematiques Appliqu\'ees, Ecole Polytechnique, Paris} \and %
Philip Protter\thanks{%
Statistics Department, Columbia University, New York, NY, 10027} 
\thanks{%
Supported in part by NSF grant DMS-0906995}} 
\date{\today}
\maketitle

\begin{abstract}
In this paper we study progressive filtration expansions with c\`adl\`ag processes. Using results from the weak convergence of $\sigma$-fields theory, we first establish a semimartingale convergence theorem. Then we apply it in a filtration expansion with a process setting and provide sufficient conditions for a semimartingale of the base filtration to remain a semimartingale in the expanded filtration. Finally, an application to the expansion of a Brownian filtration with a time reversed diffusion is given through a detailed study.
\end{abstract}

\section{Introduction}

One of the key insights of K. It\^o when he developed the It\^o integral was to restrict the space of integrands to what we now call predictable processes.  This allowed the integral to have a type of bounded convergence theorem that N. Wiener was unable to obtain with unrestricted random integrands. The It\^o integral has since been extended to general semimartingales.  If one tries however to expand (i.e. to enlarge) the filtration, then one is playing with fire, and one may lose the key properties It\^o originally obtained with his restriction to predictable processes.  In the 1980's a theory of such filtration expansions  was nevertheless successfully developed for two types of expansion:  initial expansions and progressive expansions; see for instance \cite{Jacod:1987} and \cite{Jeulin/Yor:1985}, or the more recent partial exposition in~\cite[Chapter VI]{Protter:2005}. The initial expansion of a filtration $\mathbb F = (\mathcal F_t)_{t\geq 0}$ with a random variable $\tau$ is the filtration $\mathbb H$ obtained as the right-continuous modification of $(\mathcal F_t \vee \sigma (\tau))_{t\geq 0}$. The progressive expansion $\mathbb G$ is obtained as any right-continuous filtration containing $\mathbb F$ and making $\tau$ a stopping time. When referring to the progressive expansion with a random variable in this paper, we mean the smallest such filtration. One is usually interested in the cases where $\mathbb F$~semimartingales remain semimartingales in the expanded filtrations and in their decompositions when viewed as semimartingales in the expanded filtrations.  The subject has regained interest recently, due to applications in Mathematical Finance, as exemplified (for example) in the work of Jeanblanc and Le Cam~\cite{Jeanblanc/LeCam:2009}.

In this article we go beyond the simple cases of initial expansion and progressive expansion with a random variable.  Instead we consider the (more complicated) case of expansion of a filtration through dynamic enlargement, by adding a stochastic process as it evolves simultaneously to the evolution of the original process.  In order to do this, we begin with simple cases where we add marked point processes, and then we use the theory of the convergence of sigma fields recently developed by Antonelli, Coquet, Kohatsu-Higa, Mackevicius, M\'emin, and Slominski (see~\cite{AKH},\cite{Coquet},\cite{Coquet0}) to obtain more sophisticated enlargement possibilities.  We combine the convergence results with an extension of an old result of Barlow and Protter~\cite{BarlowProtter}, finally obtaining the key results, which include the forms of the semimartingale decompositions in the enlarged filtrations.  We then apply these results to an example where we enlarge the filtration with another process which is evolving backwards in time.  To do this we need to use density estimates inspired by the work of Bally and Talay~\cite{BallyTalay}.

The techniques developed in this paper require a long preliminary treatment of the convergence of $\sigma$ fields, and to a lesser extent the convergence of filtrations.  This delays the key theorems such that they occur rather late in the paper, so perhaps it is wise to indicate that the main results of interest (in the authors' opinion) are Theorems~\ref{T:exp} and~\ref{T:semimgX}, which show how one can expand filtrations with processes and have semimartingales remain semimartingales in the enlarged filtrations.  The authors also wish to mention here that the example provided in Theorem~\ref{T:diff} shows how the hypotheses  (perhaps a bit strange at first glance) of Theorem~\ref{T:semimgX}  can arise naturally in applications, and it shows the potential utility of the results of this paper.   That said, the preliminary results on the weak convergence of $\sigma$ fields has an interest in their own right.  

\subsection{Previous Results}

For the initial expansion $\mathbb H$ of a filtration $\mathbb F = (\mathcal F_t)_{t\geq 0}$ with a random variable $\tau$, 
one well-known situation is when \textit{Jacod's criterion} is satisfied (see \cite{Jacod:1987} or alternatively~\cite[Theorem 10, p. 371]{Protter:2005}), and as far as one is concerned by the progressive expansion, filtration~$\mathbb G$, this always holds up to the random time $\tau$ as proved by Jeulin and Yor and holds on all $[0,\infty)$ for honest times (see \cite{Jeulin/Yor:1985}). In both \cite{KLP} and \cite{Jeanblanc/LeCam:2009}, this is proved to hold also for random times satisfying \textit{Jacod's criterion}. In \cite{KLP}, the authors link the two previous types of expansions and are able to provide similar results for more general types of expansion of filtrations. They extend for instance these results to the multiple time case, without any restrictions on the ordering of the individual times and more importantly to the filtration expanded by a counting process $N_t^n = \sum_{i=1}^n X_i 1_{\{\tau_i \leq t\}}$, i.e.~the smallest right-continuous filtration containing $\mathbb F$ and to which the process $N^n$ is adapted. 

For a given filtration $\mathbb F$ and a given c\`adl\`ag process $X$, the smallest right-continuous filtration containing $\mathbb F$ and to which $X$ is adapted will be called the progressive expansion of $\mathbb F$ with $X$. In this paper we pursue the analysis started in \cite{KLP} and investigate the stability of the semimartingale property of $\mathbb F$~semimartingales in progressive expansions of $\mathbb F$ with c\`adl\`ag processes $X$. We apply the results in \cite{KLP} together with results from the theory of weak convergence of $\sigma$-fields (see \cite{Coquet} and \cite{Coquet0}) to obtain a general criterion that guarantees this property, at least for $\mathbb F$~semimartingales satisfying suitable integrability assumptions. Hoover~\cite{Hoover}, following remarks by M. Barlow and S. Jacka, introduced the weak convergence of $\sigma$-fields and of filtrations in 1991.  The next big step was in 2000 with the seminal paper of Antonelli and Kohatsu-Higa\cite{AKH}.  This was quickly followed by the work of Coquet, M\'emin and Mackevicius \cite{Coquet0} and by Coquet, M\'emin and Slominsky \cite{Coquet}. We will recall fundamental results on the topic but we refer the interested reader to \cite{Coquet} and \cite{Coquet0} for details. In these papers, all filtrations are indexed by a compact time interval $[0,T]$. We work within the same framework and assume that a probability space $(\Omega, \mathcal H, P)$ and a positive integer $T$ are given. All filtrations considered in this paper are assumed to be completed by the $P$-null sets of $\mathcal H$. By the natural filtration of a process $X$, we mean the right-continuous filtration associated to the natural filtration of $X$. The concepts of weak convergence of $\sigma$-fields and of filtrations rely on the topology imposed on the space of c\`adl\`ag processes and we use the Skorohod $J_1$ topology as it is done in~\cite{Coquet}. 

\subsection{Outline}

An outline of this paper is the following. In section 2, we recall basic facts on the weak convergence of $\sigma$-fields and establish fundamental lemmas for subsequent use. The last subsection provides a sufficient condition for the semimartingale property to hold for a given c\`adl\`ag adapted process based on the weak convergence of $\sigma$-fields. The sufficient condition we provide at this point is unlikely to hold in a filtration expansion context, however the proof of this result underlines what can go wrong under the more natural assumptions considered in the next section.

Section 3 extends the main theorem in \cite{BarlowProtter} and proves a general result on the convergence of $\mathbb G^n$~special semimartingales to a $\mathbb G$~adapted process $X$, where $(\mathbb G^n)_{n\geq 1}$ and $\mathbb G$ are filtrations such that $\mathcal G^n_t$ converges weakly to $\mathcal G_t$ for each $t\geq 0$. The process $X$ is proved to be a $\mathbb G$~special semimartingale under sufficient conditions on the regularity of the local martingale and finite variation parts of the $\mathbb G^n$~semimartingales. This is then applied to the case where the filtrations $\mathbb G^n$ are obtained by progressively expanding a base filtration $\mathbb F$ with processes $N^n$ converging in probability to some process $N$. We provide sufficient conditions for an $\mathbb F$~semimartingale to remain a $\mathbb G$~semimartingale, where $\mathbb G$ is the progressive expansion of $\mathbb F$ with $N$. 

Section 4 applies the results obtained in Section 3 to the case where the base filtration $\mathbb F$ is progressively expanded by a c\`adl\`ag process whose increments satisfy a \textit{generalized Jacod's criterion} with respect to the filtration $\mathbb F$ along \textit{some} sequence of subdivisions whose mesh tends to zero. An application to the expansion of a Brownian filtration with a time reversed diffusion is given through a detailed study, and the canonical decomposition of the Brownian motion in the expanded filtration is provided. Finally, a possible application to stochastic volatility models is suggested.

\section{Weak convergence of $\sigma$-fields and filtrations} \label{S:1}

\subsection{Definitions and fundamental results}

Let $\mathbb{D}$ be the space of c\`adl\`ag\footnote{French acronym for right-continuous with left limits} functions from $[0,T]$ into $\mathbb{R}$. Let $\Lambda$ be the set of time changes from $[0,T]$ into $[0,T]$, i.e.~the set of all continuous strictly increasing functions $\lambda : [0,T]\rightarrow [0,T]$ such that $\lambda(0)=0$ and $\lambda(T)=T$. We define the Skorohod distance as follows
$$
d_S(x,y)=\inf_{\lambda \in \Lambda} \big\{ ||\lambda - Id||_{\infty} \vee ||x-y\circ\lambda||_{\infty}\big\}
$$
for each $x$ and $y$ in $\mathbb{D}$. Let $(X^n)_{n\geq 1}$ and $X$ be c\`adl\`ag processes (i.e.~whose paths are in~$\mathbb{D}$), indexed by $[0,T]$ and defined on $(\Omega, \mathcal H, P)$. We will write $X^n \stackrel{P}{\rightarrow} X$ when $(X^n)_{n\geq 1}$ converges in probability under the Skorohod $J_1$ topology to $X$ i.e.~when the sequence of random variables $(d_S(X^n,X))_{n\geq 1}$ converges in probability to zero. We can now introduce the concepts of weak convergence of $\sigma$-fields and of filtrations.

\begin{definition}\label{D:sigmafield}
A sequence of $\sigma$-algebras $\mathcal{A}^n$ converges weakly to a $\sigma$-algebra $\mathcal{A}$ if and only if for all $B\in\mathcal{A}$, $E(1_B\mid\mathcal{A}^n)$ converges in probability to $1_B$. We write $\mathcal{A}^n\stackrel{w}{\rightarrow}\mathcal{A}$.
\end{definition}

\begin{definition}\label{D:filtration}
A sequence of right-continuous filtrations $\mathbb F^n$ converges weakly to a filtration $\mathbb F$ if and only if for all $B\in\mathcal F_T$, the sequence of c\`adl\`ag martingales $E(1_B\mid\mathcal F^n_{.})$ converges in probability under the Skorohod $J_1$ topology on $\mathbb{D}$ to the martingale $E(1_B\mid\mathcal F_{.})$. We write $\mathbb F^n\stackrel{w}{\rightarrow}\mathbb F$.
\end{definition}

The following lemmas provide characterizations of the weak convergence of $\sigma$-fields and filtrations. We refer to \cite{Coquet} for the proofs.

\begin{lemma}\label{L:sigma}
A sequence of $\sigma$-algebras $\mathcal A^n$ converges weakly to a $\sigma$-algebra $\mathcal A$ if and only if $E(Z\mid\mathcal A^n)$ converges in probability to $Z$ for any integrable and $\mathcal A$~measurable random variable $Z$.
\end{lemma}

\begin{lemma}
A sequence of filtrations $\mathbb F^n$ converges weakly to a filtration $\mathbb F$ if and only if $E(Z\mid\mathcal F^n_{.})$ converges in probability under the Skorohod $J_1$ topology to $E(Z\mid\mathcal F_{.})$, for any integrable, $\mathcal F_T$~measurable random variable $Z$.
\end{lemma}

The weak convergence of the $\sigma$-fields $\mathcal F^n_t$ to $\mathcal F_t$ for all $t$ does not imply the weak convergence of the filtrations $\mathbb F^n$ to $\mathbb F$. The reverse implication does not hold neither.

Coquet, M\'emin and Slominsky provide a characterization of weak convergence of filtrations when the limiting filtration is the natural filtration of some c\`adl\`ag process $X$, see Lemma 3 in \cite{Coquet}. We provide a similar result for weak convergence of $\sigma$-fields when the limiting $\sigma$-field is generated by some c\`adl\`ag process $X$.

\begin{lemma}\label{L:carac}
Let $X$ be a c\`adl\`ag process. Define $\mathcal A=\sigma(X_t, 0\leq t\leq T)$ and let $(\mathcal A^n)_{n\geq 1}$ be a sequence of $\sigma$-fields. Then $\mathcal A^n\stackrel{w}{\rightarrow}\mathcal A$ if and only if
$$
E(f(X_{t_1}, \ldots, X_{t_k})\mid\mathcal A^n)\stackrel{P}{\rightarrow}f(X_{t_1}, \ldots, X_{t_k})
$$
for all $k\in \mathbb N$, $t_1, \ldots, t_k$ points of a dense subset $\mathcal D$ of $[0,T]$ containing $T$ and for any continuous and bounded function $f : \mathbb R^k\rightarrow \mathbb R$.
\end{lemma}

\begin{proof}
Necessity follows from  the definition of the weak convergence of $\sigma$-fields. Let us prove the sufficiency. Let $A\in\mathcal A$ and $\epsilon>0$. There exists $k\in\mathbb N$ and $t_1, \ldots, t_k$ in $\mathcal D$ such that 
$$
E(|f(X_{t_1}, \ldots, X_{t_k})-1_A|)<\epsilon.
$$
Let $\eta>0$. We need to show that $P(|E(1_A\mid\mathcal{A}^n)-1_A|\geq \eta)$ converges to zero.
\begin{align*}
P&(|E(1_A\mid\mathcal A^n)-1_A|\geq \eta)\leq P(|E(1_A\mid\mathcal A^n)-E(f(X_{t_1}, \ldots, X_{t_k})\mid\mathcal A^n)|\geq\frac{\eta}{3})\\
&+P(|E(f(X_{t_1}, \ldots, X_{t_k})\mid\mathcal A^n)-f(X_{t_1}, \ldots, X_{t_k})|\geq\frac{\eta}{3})+P(|f(X_{t_1}, \ldots, X_{t_k})-1_A|\geq\frac{\eta}{3})\\
&\leq \frac{6}{\eta}E(|f(X_{t_1}, \ldots, X_{t_k})-1_A|)+P(|E(f(X_{t_1}, \ldots, X_{t_k})\mid\mathcal A^n)-f(X_{t_1}, \ldots, X_{t_k})|\geq\frac{\eta}{3})\\
&\leq \frac{6}{\eta}\epsilon+P(|E(f(X_{t_1}, \ldots, X_{t_k})\mid\mathcal A^n)-f(X_{t_1}, \ldots, X_{t_k})|\geq\frac{\eta}{3})
\end{align*}
where the second inequality follows from the Markov inequality. By assumption, there exists $N$ such that for all $n\geq N$,
$$
P(|E(f(X_{t_1}, \ldots, X_{t_k})\mid\mathcal A^n)-f(X_{t_1}, \ldots, X_{t_k})|\geq\frac{\eta}{3})\leq \epsilon
$$
hence $P(|E(1_A\mid\mathcal A^n)-1_A|\geq \eta)\leq (\frac{6}{\eta}+1)\epsilon$.
\end{proof}

In \cite{Coquet}, the authors provide cases where the weak convergence of a sequence of natural filtrations of given c\`adl\`ag processes is guaranteed. We provide here a similar result for point wise weak convergence of the associated $\sigma$-fields.

\begin{lemma}\label{L:conv}
Let $(X^n)_{n\geq 1}$ be a sequence of c\`adl\`ag processes converging in probability to a c\`adl\`ag process $X$. Let $\mathbb F^n$ and $\mathbb F$ be the natural filtrations of $X^n$ and $X$ respectively. Then $\mathcal F^n_t\stackrel{w}{\rightarrow}\mathcal F_t$ for all $t$ such that $P(\Delta X_t \neq 0)=0$.
\end{lemma}

\begin{proof}
Let $t$ be such that $P(\Delta X_t \neq 0)=0$. Since $X$ is c\`adl\`ag, there exists $k\in\mathbb N$, and $t_1,\ldots,t_k \leq t$ such that $P(\Delta X_{t_i} \neq 0)=0$, for all $1\leq i\leq k$. Let $f : \mathbb R^k\rightarrow \mathbb R$ be a continuous and bounded function. By Lemma \ref{L:carac}, it suffices to show that
$$
E(f(X_{t_1}, \ldots, X_{t_k})\mid\mathcal F^n_t)\stackrel{P}{\rightarrow}f(X_{t_1}, \ldots, X_{t_k})
$$
An application of Markov's inequality leads to the following estimate
\begin{align*}
P(|E(f(X_{t_1}, \ldots,& X_{t_k})\mid\mathcal F^n_t)-(X_{t_1}, \ldots, X_{t_k})|\geq \eta)\\
&\leq P(|E(f(X_{t_1}, \ldots, X_{t_k})-f(X^n_{t_1}, \ldots, X^n_{t_k})\mid\mathcal F^n_t)|\geq \frac{\eta}{2})\\
&+P(|E(f(X^n_{t_1}, \ldots, X^n_{t_k})\mid\mathcal F^n_t)-f(X_{t_1}, \ldots, X_{t_k})|\geq \frac{\eta}{2})\\
&\leq \frac{4}{\eta}E(|f(X^n_{t_1}, \ldots, X^n_{t_k})-f(X_{t_1}, \ldots, X_{t_k})|)
\end{align*}
Since $X^n\stackrel{P}{\rightarrow}X$ and $P(\Delta X_{t_i} \neq 0)=0$, for all $1\leq i\leq k$, it follows that
\begin{equation}\label{eq:discr}
(X^n_{t_1}, \ldots X^n_{t_k})\stackrel{P}{\rightarrow}(X_{t_1}, \ldots X_{t_k})
\end{equation}
and hence $f(X^n_{t_1}, \ldots X^n_{t_k})$ converges in $L^1$ to $f(X_{t_1}, \ldots X_{t_k})$. This ends the proof of the lemma.
\end{proof}

For a given c\`adl\`ag process $X$, a time $t$ such that $P(\Delta X_t \neq 0)>0$ will be called a fixed time of discontinuity of $X$, and we will say that $X$ has no fixed times of discontinuity if $P(\Delta X_t\neq 0)=0$ for all $0\leq t\leq T$. Lemma \ref{L:conv} can be improved when the sequence $X^n$ is the discretization of the c\`adl\`ag process $X$ along some refining sequence of subdivisions $(\pi_n)_{n\geq 1}$ such that each fixed time of discontinuity of $X$ belongs to $\cup_n \pi_n$.

\begin{lemma}\label{L:discr}
Let $X$ be a c\`adl\`ag process. Consider a sequence of subdivisions $(\pi_n=\{t^n_k\}, n\geq~1)$ whose mesh tends to zero and let $X_n$ be the discretized process defined by $X_t^n = X_{t^n_k}$, for all $t_n^k \leq t < t^n_{k+1}$. Let $\mathbb F$ and $\mathbb F^n$ be the natural filtrations of $X$ and $X^n$. If each fixed time of discontinuity of $X$ belongs to $\cup_n \pi_n$, then $\mathcal F^n_t\stackrel{w}{\rightarrow}\mathcal F_t$, for all $t$.
\end{lemma}

\begin{proof}
The proof is essentially the same as that of Lemma \ref{L:conv}. Now, equation (\ref{eq:discr}) holds because the subdivision contains the discontinuity points of $X$. 
\end{proof}

We will also need the two following lemmas from the theory of weak convergence of $\sigma$-fields. The first result is proved in \cite{Coquet0} and the second one in \cite{Coquet}.

\begin{lemma}\label{L:vee}
Let $(\mathcal A^n)_{n\geq 1}$ and $(\mathcal B^n)_{n\geq 1}$ be two sequences of $\sigma$-fields that weakly converge to $\mathcal A$ and $\mathcal B$, respectively. Then
$$
\mathcal A^n\vee\mathcal B^n\stackrel{w}{\rightarrow}\mathcal A\vee\mathcal B
$$
\end{lemma}

\begin{lemma}\label{L:subset}
Let $(\mathcal A^n)_{n\geq 1}$ and $(\mathcal B^n)_{n\geq 1}$ be two sequences of $\sigma$-fields such that $\mathcal A^n \subset \mathcal B^n$ for all $n$. Let $\mathcal A$ be a $\sigma$-field. If $\mathcal A^n\stackrel{w}{\rightarrow}\mathcal A$ then $\mathcal B^n\stackrel{w}{\rightarrow}\mathcal A$.
\end{lemma}

As pointed out in \cite{Coquet}, the results in Lemmas \ref{L:vee} and \ref{L:subset} are not true as far as one is interested in weak convergence of filtrations.

\subsection{Approximation of a given stopping time}

Let $(\mathbb{G}^n)_{n\geq 1}$ be a sequence of right-continuous filtrations and let $\mathbb{G}$ be a right-continuous filtration such that $\mathcal{G}^n_t\stackrel{w}{\rightarrow}\mathcal{G}_t$ for all $t$. In order to obtain our filtration expansion results, we need a key theorem that guarantees the $\mathbb G$~semimartingale property of a limit of $\mathbb G^n$~semimartingales as in Theorem \ref{T:semimg2}. The following lemma, which permits to approximate any $\mathbb G$~bounded stopping time $\tau$ by a sequence of $\mathbb G^n$~stopping times, will be of crucial importance in the proof of Theorem \ref{T:semimg2}, Part (ii). We prove this result using successive approximations in the case where $\tau$ takes a finite number of values and show how this property is inherited by bounded stopping times. We do not study the general case (unbounded stopping times) since we are working on the finite time interval $[0,T]$.

\begin{lemma}\label{L:stop}
Let $(\mathbb{G}^n)_{n\geq 1}$ be a sequence of right-continuous filtrations and let $\mathbb{G}$ be a right-continuous filtration such that $\mathcal{G}^n_t\stackrel{w}{\rightarrow}\mathcal{G}_t$ for all $t$. Let $\tau$ be a bounded $\mathbb G$ stopping time. Then there exists $\phi : \mathbb{N}\rightarrow \mathbb{N}$ strictly increasing and a bounded sequence $(\tau_n)_{n\geq 1}$ such that the subsequence $(\tau_{\phi(n)})_{n\geq 1}$ converges in probability to $\tau$ and each $\tau_{\phi(n)}$ is a $\mathbb G^{\phi(n)}$~stopping time.
\end{lemma}

\begin{proof}
Let $\tau$ be a $\mathbb G$~stopping time bounded by $T$. Then there exists a sequence $\tau_n$ of $\mathbb G$ stopping times decreasing a.s. to $\tau$ and taking values in $\big\{\frac{k}{2^n}, k\in\{0, 1, \cdots, [2^nT]+1\}\big\}$. This is true since the sequence $\tau_n=\frac{[2^n\tau]+1}{2^n}$ obviously works. Hence $\tau_n$ takes a finite number of values. We claim that 

\textbf{Claim.} \textit{for each $n$, we can construct a sequence $(\tau_{n,m})_{m\geq 1}$ converging in probability to~$\tau_n$, and such that $\tau_{n,m}$ is a $\mathbb G^m$~stopping time, for each $m$.}

Assume we can do so and let $\eta>0$ and $\epsilon>0$. Then for each $n$, $\lim_{m\rightarrow\infty}P(|\tau_{n,m}-\tau_n|>\frac{\eta}{2})=0$, i.e.~for each $n$ there exists $M_n$ such that for all $m\geq M_n$, $P(|\tau_{n,m}-\tau_n|>\frac{\eta}{2})\leq\frac{\epsilon}{2}$. Define $\phi(1)=M_1$ and $\phi(n)=\max(M_n, \phi(n-1)+1)$ by induction. The application $\phi:\mathbb N\rightarrow\mathbb N$ is strictly increasing, and for each $n$, $\tau_{n,\phi(n)}$ is a $\mathbb G^{\phi(n)}$~stopping time and $P(|\tau_{n,\phi(n)}-\tau_n|>\frac{\eta}{2})\leq\frac{\epsilon}{2}$. It follows that
$$
P(|\tau_{n,\phi(n)}-\tau|>\eta)\leq P(|\tau_{n,\phi(n)}-\tau_n|>\frac{\eta}{2})+P(|\tau_n-\tau|>\frac{\eta}{2})\leq \frac{\epsilon}{2}+P(|\tau_n-\tau|>\frac{\eta}{2})
$$
Since $\tau_n$ converges to $\tau$, there exists some $n_0$, such that for all $n\geq n_0$, $P(|\tau_n-\tau|>\frac{\eta}{2})\leq \frac{\epsilon}{2}$. Hence $\tau_{n,\phi(n)}\stackrel{P}{\rightarrow}\tau$. So in order to prove the lemma, it only remains to prove the claim above. 

\textbf{Proof of the claim.} We drop the index $n$ and assume that $\tau$ is a $\mathbb G$~stopping time that takes a finite number of values $t_1,\cdots, t_M$.  Since $\mathbb G$ is right-continuous, $1_{\{\tau=t_i\}}$ is $\mathcal{G}_{t_i}$~measurable, and since by assumption, for all $i$, $\mathcal{G}^m_{t_i}\stackrel{w}{\rightarrow}\mathcal{G}_{t_i}$, it follows that for all $i$
$$
E(1_{\{\tau=t_i\}}\mid\mathcal{G}^m_{t_i})\stackrel{P}{\rightarrow}1_{\{\tau=t_i\}}
$$
Now for $i=1$, we can extract a subsequence $E(1_{\{\tau=t_1\}}\mid\mathcal{G}^{\phi_1(m)}_{t_1})$ converging to $1_{\{\tau=t_1\}}$~a.s. and any sub-subsequence will also converge to $1_{\{\tau=t_1\}}$ a.s. Also, $\mathcal{G}^{\phi_1(m)}_{t_2}\stackrel{w}{\rightarrow}\mathcal{G}_{t_2}$, hence $E(1_{\{\tau=t_2\}}\mid\mathcal{G}^{\phi_1(m)}_{t_2})\stackrel{P}{\rightarrow}1_{\{\tau=t_2\}}$, and we can extract a further subsequence $E(1_{\{\tau=t_2\}}\mid\mathcal{G}^{\phi_1(\phi_2(m))}_{t_2})$ that converges a.s. to $1_{\{\tau=t_2\}}$. Since we have a finite number of possible values, we can repeat this reasoning up to time $t_M$. Define then $\phi=\phi_1 \circ \phi_2 \circ \cdots \circ \phi_n$, we get for all $i\in\{1,\cdots, M\}$, 
$$
E(1_{\{\tau=t_i\}}\mid\mathcal{G}^{\phi(m)}_{t_i})\stackrel{a.s}{\rightarrow}1_{\{\tau=t_i\}}.
$$
Define $\tau_m=\min_{\{i\mid E(1_{\{\tau=t_i\}}\mid\mathcal{G}^{m}_{t_i})>\frac{1}{2}\}}t_i$. Then
$$
\{\tau_m=t_i\}=\{E(1_{\{\tau=t_i\}}\mid\mathcal{G}^{m}_{t_i})>\frac{1}{2}\}\cap\{\forall t_j<t_i, E(1_{\{\tau=t_j\}}\mid\mathcal{G}^{m}_{t_j})\leq\frac{1}{2}\}
$$
and hence $\tau_m$ is a $\mathbb G^m$~stopping time. Also, obviously, $\tau_{\phi(m)}\stackrel{a.s}{\rightarrow}\tau$, hence $\tau_m\stackrel{P}{\rightarrow}\tau$.
\end{proof}

\subsection{Weak convergence of $\sigma$-fields and the semimartingale property}

Assume we are given a sequence of filtrations $(\mathbb F^m)_{m\geq 1}$ and define the filtration $\tilde{\mathbb F}=(\tilde{\mathcal F}_t)_{0\leq t\leq T}$, where $\tilde{\mathcal F}_t=\bigvee_{m}\mathcal F^m_t$. We prove in this section a stability result for $\tilde{\mathbb F}$~semimartingales. More precisely, we prove that if $X$ is an $\tilde{\mathbb F}$~semimartingale, then it remains an $\mathbb F$~semimartingale for any limiting (in the sense $\mathcal F^m_t\stackrel{w}{\rightarrow}\mathcal F_t$, for all $t \in [0,T]$) filtration $\mathbb F$ to which it is adapted.

The crucial tool for proving our first theorem is the Bichteler-Dellacherie characterization of semimartingales (see for example~\cite{Protter:2005}). Recall that if $\mathbb H$ is a filtration, an $\mathbb H$~predictable elementary process $H$ is a process of the form
$$
H_t(\omega) = \sum_{i=1}^{k}h_i(\omega)1_{]t_i,t_{i+1}]}(t);
$$
where $0 \leq t_1 \leq \ldots \leq t_{k+1} < \infty$, and each $h_i$ is $\mathcal H_{t_i}$~measurable. Moreover, for any $\mathbb H$~adapted c\`adl\`ag process $X$ and predictable elementary process $H$ of the above form, we write
$$
J_X(H) = \sum_{i=1}^{k} h_i(X_{t_{i+1}} - X_{t_{i}})
$$
\begin{theorem}[Bichteler-Dellacherie]\label{T:BichDell}
Let $X$ be an $\mathbb H$~adapted c\`adl\`ag process. Suppose that for every sequence $(H_n)_{n\geq 1}$ of bounded, $\mathbb H$ predictable elementary processes that are null outside a fixed interval $[0,N]$ and convergent to zero uniformly in $(\omega; t)$, we have that $\lim_{n\rightarrow\infty}J_X(H_n) = 0$ in probability. Then $X$ is an $\mathbb H$~semimartingale.
\end{theorem}

The converse is true by the Dominated Convergence Theorem for stochastic integrals. We can now state and prove the main theorem of this subsection.

\begin{theorem}\label{T:semimg}
Let $(\mathbb F^m)_{m\geq 1}$ be a sequence of filtrations. Let $\mathbb F$ be a filtration such that for all $t \in [0,T]$, $\mathcal F^m_t\stackrel{w}{\rightarrow}\mathcal F_t$. Define the filtration $\tilde{\mathbb F}=(\tilde{\mathcal F}_t)_{0\leq t\leq T}$, where $\tilde{\mathcal F}_t=\bigvee_{m}\mathcal F^m_t$. Let $X$ be an $\mathbb F$~adapted c\`adl\`ag process such that $X$ is an $\tilde{\mathbb F}$~semimartingale. Then $X$ is an $\mathbb F$~semimartingale.
\end{theorem}

\begin{proof}
For a fixed $N>0$, consider a sequence of bounded, $\mathbb F$ predictable elementary processes of the form
$$
H^n_t = \sum_{i=1}^{k_n}h_i^n1_{]t^n_i,t^n_{i+1}]}(t);
$$
null outside the fixed time interval $[0,N]$ and with $h^n_i$ being $\mathcal F_{t^n_i}$~measurable. Suppose that $H^n$ converges to zero uniformly in $(\omega, t)$. We prove that $ J_X(H^n) \stackrel{P}{\rightarrow} 0$.

For each $m$, define the sequence of bounded $\mathbb F^m$~predictable elementary processes
$$
H^{n,m}_t = \sum_{i=1}^{k_n}E(h_i^n\mid\mathcal{F}^m_{t^n_i})1_{]t^n_i,t^n_{i+1}]}(t);
$$
By assumption, $\mathcal F^m_t\stackrel{w}{\rightarrow}\mathcal F_t$ for all $0\leq t\leq T$. Hence for all $n$ and $1\leq i\leq k_n$, $\mathcal F^m_{t_i^n}\stackrel{w}{\rightarrow}\mathcal F_{t_i^n}$. Since $h^n_i$ is bounded (hence integrable) and $\mathcal F_{t_i^n}$~measurable, it follows from Lemma \ref{L:sigma} that $E(h^n_i\mid\mathcal F^m_{t^n_i})\stackrel{P}{\rightarrow}h^n_i$ and hence $E(h^n_i\mid\mathcal F^m_{t^n_i})(X_{t_{i+1}^n}-X_{t_i^n})\stackrel{P}{\rightarrow}h^n_i(X_{t^n_{i+1}}-X_{t^n_i})$ for each $n$ and $1\leq i\leq k_n$ since $(X_{t^n_{i+1}}-X_{t^n_i})$ is finite~a.s. 
Let $\eta>0$.
$$
P\Big(\big|J_X(H^{n,m})-J_X(H^n)\big|>\eta\Big) \leq \sum_{i=1}^{k_n}P\Big(\Big|\big(E(h^n_i\mid\mathcal F^m_{t^n_i})-h^n_i\big)\big(X_{t^n_{i+1}}-X_{t^n_i}\big)\Big|>\frac{\eta}{k_n}\Big)
$$
For each fixed $n$, the right side quantity converges to $0$ as $m$ tends to $\infty$. This proves that for each $n$,
$$
J_X(H^{n,m})\stackrel{P}{\rightarrow}J_X(H^n).
$$
Let $\delta>0$ and $\epsilon>0$. For each $n$ and $m$,
\begin{equation}\label{eq:exchangelim}
P(|J_X(H^n)|>\delta)\leq P\Big(\big|J_X(H^{n,m})-J_X(H^n)\big|>\frac{\delta}{2}\Big)+P(|J_X(H^{n,m})|>\frac{\delta}{2})
\end{equation}
From $J_X(H^{n,m})\stackrel{P}{\rightarrow}J_X(H^n)$, it follows that for each $n$, there exists $M^n_0$ such that for all $m\geq M^n_0$,
$$
P\Big(\big|J_X(H^{n,m})-J_X(H^n)\big|>\frac{\delta}{2}\Big)\leq \frac{\epsilon}{2}
$$
Hence $P(|J_X(H^n)|>\delta)\leq \frac{\epsilon}{2}+P(|J_X(H^{n,M^n_0})|>\frac{\delta}{2})$. First $E(h_i^n\mid\mathcal F^{M^n_0}_{t^n_i})$ is bounded, $\tilde{\mathcal F}_{t^n_i}$~measurable so that $H_t^{n,M^n_0}=\sum_{i=1}^{k_n}E(h_i^n\mid\mathcal F^{M^n_0}_{t^n_i})1_{]t^n_i,t^n_{i+1}]}(t)$ is a bounded $\tilde{\mathbb F}$~predictable process. Since $H^n$ converges to zero uniformly in $(\omega, t)$, it follows that $h^n_i$ converges to zero uniformly in $(\omega, i)$ so that there exists $n_0$ such that for each $n\geq n_0$, for all $(\omega, i)$, $|h^n_i(\omega)|\leq \epsilon$. Hence, for all $(\omega, t)$ and $n\geq n_0$
$$
|H^{n,M^n_0}_t(\omega)| \leq \sum_{i=1}^{k_n}E(|h_i^n|\mid\mathcal{F}^{M^n_0}_{t^n_i})(\omega)1_{]t^n_i,t^n_{i+1}]}(t)\leq \epsilon \sum_{i=1}^{k_n}1_{]t^n_i,t^n_{i+1}]}(t) \leq \epsilon
$$
Therefore $H^{n,M^n_0}$ is a sequence of bounded $\tilde{\mathbb F}$~predictable processes null outside the fixed interval $[0,N]$ that converges uniformly to zero in $(\omega, t)$. Since by assumption $X$ is a $\tilde{\mathbb F}$~semimartingale, it follows from the converse of Bichteler-Dellacherie's theorem that $J_X(H^{n,M^n_0})$ converges to zero in probability, hence, for $n$ large enough, $P(|J_X(H^{n,M^n_0})|>\frac{\delta}{2})\leq \frac{\epsilon}{2}$ and
$$
P(|J_X(H^n)|>\delta)\leq\epsilon
$$
Applying now Theorem \ref{T:BichDell} proves that $X$ is an $\mathbb F$~semimartingale.
\end{proof}

Let $X$ be an $\tilde{\mathbb F}$~semimartingale. Theorem \ref{T:semimg} proves that $X$ remains an $\mathbb F$~semimartingale for any limiting filtration $\mathbb F$ (in the sense $\mathcal F^m_t \stackrel{w}{\rightarrow} \mathcal F_t$ for all $0\leq t\leq T$) to which $X$ is adapted. Of course, if $\mathbb F\subset\tilde{\mathbb F}$, Stricker's theorem already implies that $X$ is an $\mathbb F$~semimartingale. But there is no general link between the filtration $\tilde{\mathbb F}=\bigvee_m\mathbb F^m$ and the limiting filtration~$\mathbb F$. A trivial example is given by taking $\mathbb F$ to be the trivial filtration (it can be seen from Definition \ref{D:sigmafield} that the trivial filtration satisfies $\mathcal F^m_t\stackrel{w}{\rightarrow}\mathcal F_t$, for all $t$, for any given sequence of filtrations $\mathbb F^m$). One can also have $\bigvee_m\mathbb F^m\subset\mathbb F$, as it is the case in the following important example.

\begin{example}\label{ex:dis}
Let $X$ be a c\`adl\`ag process. Consider a sequence of subdivisions $\{t_k^n\}$ whose mesh tends to zero and let $X^n$ be the discretized process defined by $X^n_t=X_{t^n_k}$, for all $t^n_k\leq t <t^n_{k+1}$. Let $\mathbb F$ and $\mathbb F^n$ be the natural filtrations of $X$ and $X^n$. It is well known that for all $t$, $\mathcal F_{t^{-}} \subset \bigvee_{n}\mathcal F^n_t\subset\mathcal F_t$. Also, $X^n$ converges a.s.~to the process $X$, hence $X^n\stackrel{P}{\rightarrow}X$. Assume now that $X$ has no fixed times of discontinuity. Then Lemma \ref{L:conv} guarantees that $\mathcal F^n_t\stackrel{w}{\rightarrow}\mathcal F_t$, for all $t$. Moreover, if $\mathbb F$ is left-continuous (which is usually the case, and holds for example  when $X$ is a c\`adl\`ag Hunt Markov process) then $\bigvee_{n}\mathcal F^n_t=\mathcal F_t$ for all $t$.
\end{example}

We provide now another example where $\bigvee_n\mathcal F^n_t$ is itself a limiting $\sigma$-field for $(\mathcal F^n_t)_{n\geq 1}$, for each~$t$.

\begin{example}
Assume that $\mathbb F^n$ is a sequence of filtrations such that for all $t$, the sequence of $\sigma$-fields $(\mathcal F^n_t)_{n\geq 1}$ is increasing for the inclusion. Define $\tilde{\mathcal F}_t=\bigvee_n\mathcal F^n_t$. Then for each $t$, $\mathcal F^n_t\stackrel{w}{\rightarrow}\tilde{\mathcal F}_t$. To see this, fix $t$ and let $X$ be an integrable $\tilde{\mathcal F}_t$~measurable random variable. Then $M_n=E(X\mid\mathcal F^n_t)$ is a closed martingale and the convergence theorem for closed martingales ensures that $M_n$ converges to $X$ in $L^1$, which implies that $E(X\mid\mathcal F^n_t)\stackrel{P}{\rightarrow}X$. Lemma \ref{L:sigma} allows us to conclude.
\end{example}

Checking in practice that X is an $\tilde{\mathbb F}$~semimartingale can be a hard task. In subsequent sections, we replace the strong assumption \textit{$X$ is an $\tilde{\mathbb F}$~semimartingale} by the more natural assumption \textit{$X$ is an $\mathbb F^n$~semimartingale, for each $n$}. Theorem \ref{T:semimg} is very instructive since we see from the proof what goes wrong under this new assumption : the change in the order of limits in (\ref{eq:exchangelim}) cannot be justified anymore and extra integrability conditions will be needed. They are introduced in the next section. 

This assumption arises naturally in filtration expansion theory in the following way. Assume we are given a base filtration $\mathbb F$ and a sequence of processes $N^n$ which converges (in probability for the Skorohod $J_1$ topology) to some process $N$. Let $\mathbb N^n$ and $\mathbb N$ be their natural filtrations and $\mathbb G^n$ (resp. $\mathbb G$) the smallest right-continuous filtration containing $\mathbb{F}$ and to which $N^n$ (resp. $N$) is adapted. Assume that for each $n$, every $\mathbb F$~semimartingale remains a $\mathbb G^n$~semimartingale. Does this property also hold between $\mathbb F$ and $\mathbb G$? In the next section we answer this question under the assumption of weak convergence of the $\sigma$-fields $\mathcal G^n_t$ to $\mathcal G_t$ for each $t$, for a class of $\mathbb F$~semimartingales $X$ satisfying some integrability conditions. If moreover $\mathbb G^n\stackrel{w}{\rightarrow}\mathbb G$, we are able to provide the $\mathbb G$~decomposition of such~$X$. 

\section{Filtration expansion with processes}

In preparation for treating the expansion of filtrations via processes, we need to establish a general result on the convergence of semimartingales, which is perhaps of interest in its own right.  

\subsection{Convergence of semimartingales}

The following theorem is a generalization of the main result in \cite{BarlowProtter}.

\begin{theorem}\label{T:semimg2}
Let $(\mathbb{G}^n)_{n\geq 1}$ be a sequence of right-continuous filtrations and let $\mathbb{G}$ be a filtration such that $\mathcal{G}^n_t\stackrel{w}{\rightarrow}\mathcal{G}_t$ for all $t$. Let $(X^n)_{n\geq 1}$ be a sequence of $\mathbb{G}^n$~semimartingales with canonical decomposition $X^n = X^n_0 + M^n + A^n$. Assume there exists $K>0$ such that for all $n$, 
$$
E(\int_0^T|dA^n_s|) \leq K \qquad \text{and} \qquad E(\sup_{0\leq s\leq T}|M^n_s|)\leq K
$$
Then the following holds.
\begin{itemize}
	\item[(i)] Assume there exists a $\mathbb{G}$~adapted process $X$ such that $E(\sup_{0\leq s\leq T}|X^n_s-X_s|)\rightarrow 0$. Then $X$ is a $\mathbb{G}$~special semimartingale.
	
	\item[(ii)] Moreover, assume $\mathbb G$ is right-continuous and let $X=M+A$ be the canonical decomposition of $X$. Then $M$ is a $\mathbb{G}$ martingale and $\int_0^T|dA_s|$ and $\sup_{0\leq s\leq T} |M_s|$ are integrable.
\end{itemize}
\end{theorem}

\begin{proof}\textbf{Part (i).} The idea of the proof of Part (i) is similar to the one in \cite{BarlowProtter}. First, $X$ is c\`adl\`ag since it is the a.s.~uniform limit of a subsequence of the c\`adl\`ag processes $(X^n)_{n\geq 1}$. Also, since $||X^n_0-X_0||_1\rightarrow 0$, we can take w.l.o.g $X_0^n=X_0=0$, and we do~so. The integrability assumptions guarantee that $E(\sup_s|X^n_s|)\leq 2K$ and up to replacing $K$ by $2K$, we assume that $E(\sup_s|M^n_s|)\leq K$, $E(\int_0^T|dA^n_s|)\leq K$ and $E(\sup_s|X^n_s|)\leq K$. Then $E(\sup_s|X_s|)\leq E(\sup_s|X_s-X^n_s|)+K$ and by taking limits $E(\sup_s|X_s|)\leq K$.

Let $H$ be a $\mathbb{G}$~predictable elementary process of the form $H_t=\sum_{i=1}^kh_i1_{]t_i,t_{i+1}]}(t)$, where $h_i$ is a $\mathcal{G}_{t_i}$~measurable random variable such that $|h_i|\leq 1$ and $t_1<\ldots <t_k<t_{k+1}=T$. Define now $H^n_t=\sum_{i=1}^kh^n_i1_{]t_i,t_{i+1}]}(t)$, where $h^n_i=E(h_i\mid\mathcal{G}^n_{t_i})$. Then $h^n_i$ is a $\mathcal{G}^n_{t_i}$~measurable random variable satisfying $|h^n_i|\leq 1$, hence $H^n$ is a bounded $\mathbb{G}^n$~predictable elementary process. It follows that $H^n\cdot M^n$ is a $\mathbb{G}^n$~martingale and for each $n$,
$$
|E((H^n\cdot X^n)_T)| \leq |E(\int_0^TH^n_sdA^n_s)|\leq E(\int_0^{T}|dA^n_s|)\leq K
$$
Therefore, for each $n$,
\begin{equation}\label{eq:lim}
|E((H\cdot X)_T)|\leq |E((H \cdot X)_T-(H^n\cdot X^n)_T)|+K
\end{equation}
Since $h_i$ is $\mathcal{G}_{t_i}$~measurable and $\mathcal{G}_t^n\stackrel{w}{\rightarrow}\mathcal{G}_t$ for all $t$, $h^n_i\stackrel{P}{\rightarrow}h_i$ for all $1\leq i\leq k$. Since the set $\{1,\ldots,k\}$ is finite, successive extractions allow us to find a subsequence $\psi(n)$ (independent from $i$) such that for each $1\leq i\leq k$, $h^{\psi(n)}_i$ converges a.s.~to $h_i$. So up to working with the $\mathbb{G}^{\psi(n)}$ predictable elementary processes $H^{\psi(n)}$ and the stochastic integrals $H^{\psi(n)}\cdot X^{\psi(n)}$ in (\ref{eq:lim}), we can assume that $h^n_i$ converges a.s.~to $h_i$, for each $1\leq i\leq k$.

Now, $|E((H\cdot X)_T-(H^n\cdot X^n)_T)|\leq \sum_{i=1}^kE(|h_iY_i-h^n_iY^n_i|)$ where $Y_i=X_{t_{i+1}}-X_{t_i}$ and $Y^n_i=X^n_{t_{i+1}}-X^n_{t_i}$. Each term in the sum can be bounded as follows.
\begin{align*}
E(|h_iY_i-&h^n_iY^n_i|)\leq E(|Y^n_i(h^n_i-h_i)|)+E(|h_i(Y^n_i-Y_i)|)\\
&\leq 2E(\sup_s|X^n_s| |h^n_i-h_i|) + E(|Y^n_i-Y_i|)\\
&\leq 2E(\sup_s|X^n_s-X_s| |h^n_i-h_i|) + 2E(\sup_s|X_s| |h^n_i-h_i|) + 2E(\sup_s|X^n_s-X_s|)\\
&\leq 6E(\sup_s|X^n_s-X_s|) + 2E(\sup_s|X_s| |h^n_i-h_i|)
\end{align*}
Since $\sup_s|X_s| |h^n_i-h_i|$ converges a.s to zero and that for all $n$, $|h^n_i|\leq 1$, hence $\sup_s|X_s| |h^n_i-h_i|\leq 2\sup_s|X_s|$ and $\sup_s|X_s|\in L^1$, the Dominated Convergence Theorem implies that $E(\sup_s|X_s| |h^n_i-h_i|)\rightarrow 0$. Since by assumption $E(\sup_s|X^n_s-X_s|)\rightarrow 0$, it follows that $|E((H\cdot X)_T-(H^n\cdot X^n)_T)|$ converges to $0$. Letting $n$ tend to infinity in (\ref{eq:lim}) gives $|E((H\cdot X)_T)| \leq K$. So $X$ is a $\mathbb{G}$~quasimartingale, hence a $\mathbb{G}$~special semimartingale. Therefore $X$ has a $\mathbb G$ canonical decomposition $X=M+A$ where M is a $\mathbb G$~local martingale and $A$ is a $\mathbb G$~predictable finite variation process.

\textbf{Part(ii).} Let $(\tau_m)_{m\geq 1}$ be a sequence of bounded $\mathbb G$~stopping times that reduces $M$. Since for all $t$, $\mathcal{G}^m_t\stackrel{w}{\rightarrow}\mathcal{G}_t$, it follows from Lemma \ref{L:stop} that for each $m$ there exist a function $\phi_m$ strictly increasing and a sequence $(\tau_m^n)_{n\geq 1}$ such that $(\tau_m^{\phi_m(n)})_{n\geq 1}$ converges in probability to $\tau_m$ and $\tau_m^{\phi_m(n)}$ are bounded $\mathbb G^{\phi_m(n)}$ stopping times. We can extract a subsequence $(\tau_m^{\phi_m(\psi_m(n))})_{n\geq 1}$ converging a.s.~to $\tau_m$. In order to simplify the notation, fix $m\geq 1$ and up to working with $\tilde{\mathbb G}^n=\mathbb G^{\phi_m(\psi_m(n))}$ instead of $\mathbb G^n$ (which satisfies the same assumptions), take $\Phi_m := \phi_m \circ \psi_m$ to be the identity. Let $H$ be a $\mathbb G$~elementary predictable process as defined in Part (i). Since $\tau_m$ reduces $M$, $E((H\cdot A)_{\tau_m}) = E((H\cdot X)_{\tau_m})$. We can write
$$
E((H\cdot A)_{\tau_m})=E\big((H\cdot X)_{\tau_m}-(H^n\cdot X^n)_{\tau_m}\big) + E\big((H^n\cdot X^n)_{\tau_m}-(H^n\cdot X^n)_{\tau_m^n}\big) + E\big((H^n\cdot X^n)_{\tau_m^n}\big)
$$

We start with the third term. Since $H^n\cdot M^n$ is a $\mathbb G^n$~martingale and $\tau_m^n$ is a bounded $\mathbb G^n$~stopping time, it follows from Doob's optional sampling theorem that $E((H^n\cdot X^n)_{\tau_m^n}) = E((H^n\cdot A^n)_{\tau_m^n})$, hence $|E((H^n\cdot X^n)_{\tau_m^n})|\leq E(\int_0^{\tau_m}|dA^n_s|)\leq K$.

We focus now on the first term. Let $Y^i_s=X_s-X_{t_i}$ and $Y^{i,n}_s=X^n_s-X^n_{t_i}$.
\begin{align*}
E_1 &:= |E\big((H\cdot X)_{\tau_m}-(H^n\cdot X^n)_{\tau_m}\big)|\leq E\Big(\sup_{0\leq s\leq T}\big|(H\cdot X)_{s}-(H^n\cdot X^n)_{s}\big|\Big)\\
&\leq E\Big(\sum_{i=1}^k\sup_{t_i<s\leq t_{i+1}}\big|h_iY^i_s-h_i^nY^{i,n}_s\big|\Big)\\
&\leq \sum_{i=1}^k E\Big(\sup_{t_i<s\leq t_{i+1}} |Y^{i,n}_s||h^n_i-h_i| + \sup_{t_i<s\leq t_{i+1}} |h_i||Y^{i,n}_s-Y^i_s|\Big)
\end{align*}
Since $|h^i|\leq 1$, $|Y^{i,n}_s|\leq 2 \sup_{u}|X^n_u|$ and $|Y^{i,n}_s-Y^i_s|\leq 2 \sup_{u}|X^n_u-X_u|$, it follows that
\begin{align*}
E_1 &\leq 2 \sum_{i=1}^k \big\{ E\big(\sup_{0\leq u\leq T} |X^n_s||h^n_i-h_i|\big) + E\big(\sup_{0\leq u\leq T} |X^n_u-X_u|\big)\big\}\\
&\leq 6kE(\sup_{0\leq s\leq T}|X^n_s-X_s|)+2\sum_{i=1}^k\{E\big(\sup_{0\leq s\leq T}|X_s||h^n_i-h_i|\big)\}
\end{align*}

We study now the second term $E_2:=E\big((H^n\cdot X^n)_{\tau_m}-(H^n\cdot X^n)_{\tau_m^n}\big)$. Let $0<\eta<\min_i |t_{i+1}-t_i|$, and define $Y^n=H^n\cdot X^n$. Write now
$$
E\big(|Y^n_{\tau_m}-Y^n_{\tau^n_m}|\big)=E\big(|Y^n_{\tau_m}-Y^n_{\tau^n_m}|1_{\{|\tau_m-\tau_m^n|\leq \eta\}}\big)+E\big(|Y^n_{\tau_m}-Y^n_{\tau^n_m}|1_{\{|\tau_m-\tau_m^n|> \eta\}}\big)=: e_1 + e_2
$$
We study each of the two terms separately. We start with $e_2$.
\begin{align*}
e_2&\leq E\big((|Y^n_{\tau_m}|+|Y^n_{\tau^n_m}|)1_{\{|\tau_m-\tau_m^n|> \eta\}}\big)\leq 2 E\Big(\sum_{i=1}^k\sup_{t_i<s\leq t_{i+1}}|X^n_s-X^n_{t_i}|1_{\{|\tau_m-\tau_m^n|> \eta\}}\Big)\\
&\leq 4 k E\Big(\sup_{0\leq u\leq T}|X^n_u|1_{\{|\tau_m-\tau_m^n|> \eta\}}\Big)\\
&\leq 4 k E\Big(\sup_{0\leq u\leq T}|X^n_u-X_u|\Big)+4 k E\Big(\sup_{0\leq u\leq T}|X_u|1_{\{|\tau_m-\tau_m^n|> \eta\}}\Big)
\end{align*}
We study now $e_1$. On $\{|\tau_m-\tau^n_m|\leq \eta\}$ and since $\eta<\min_i|t_{i+1}-t_i|$, we have $|Y^n_{\tau_m}-Y^n_{\tau^n_m}|\leq 2\sup_{s\leq t\leq s+\eta}|X^n_t-X^n_s|$. In fact, one of the two following cases is possible for $\tau_m$ and $\tau_m^n$. Either they are both in the same interval $(t_i, t_{i+1}]$, in which case, $|Y^n_{\tau_m}-Y^n_{\tau^n_m}|=|h^n_i(X^n_{\tau_m}-X^n_{\tau_m^n})|\leq \sup_{s\leq t\leq s+\eta}|X^n_t-X^n_s|$, or they are in two consecutive intervals. For the second case, take for example $t_{i-1}<\tau_m\leq t_i<\tau_m^n\leq t_{i+1}$, then 
\begin{align*}
|Y^n_{\tau_m}-Y^n_{\tau^n_m}|&=|h^n_{i-1}(X^n_{\tau_m}-X^n_{t_{i-1}})-h^n_{i-1}(X^n_{t_i}-X^n_{t_{i-1}})-h^n_i(X^n_{\tau_m^n}-X^n_{t_i})|\\
&=|h^n_{i-1}(X^n_{\tau_m}-X^n_{t_i})-h^n_i(X^n_{\tau_m^n}-X^n_{t_i})|\leq |X^n_{\tau_m}-X^n_{t_i}|+|X^n_{\tau_m^n}-X^n_{t_i}|\\
&\leq 2\sup_{s\leq t\leq s+\eta}|X^n_t-X^n_s|
\end{align*}
The case $t_{i-1}<\tau_m^n\leq t_i<\tau_m\leq t_{i+1}$ is similar. Hence
$$
e_1\leq 2 E(\sup_{s\leq t\leq s+\eta}|X^n_t-X^n_s|1_{\{|\tau_m-\tau_m^n|\leq \eta\}})\leq 2 E(\sup_{s\leq t\leq s+\eta}|X^n_t-X^n_s|)
$$
Putting all this together yields for each $0<\eta<\min_{i}|t_{i+1}-t_i|$ and each $n$
\begin{align*}
|E(H\cdot A)_{\tau_m}|\leq K + &2 E(\sup_{s\leq t\leq s+\eta}|X^n_t-X^n_s|) + 2\sum_{i=1}^kE(\sup_{0\leq s\leq T}|X_s||h^n_i-h_i|)\\
&+ 4kE(\sup_{0\leq s\leq T}|X_s|1_{\{|\tau_m-\tau_m^n|>\eta\}}) + 10 k E\big(\sup_{0\leq u\leq T} |X^n_u-X_u|\big)
\end{align*}
Getting back to the general case, we obtain for each $m\geq 1$ and each $n\geq 1$,
\begin{align*}
|E(H\cdot A)_{\tau_m}|\leq K + &2 E(\sup_{s\leq t\leq s+\eta}|X^{\Phi_m(n)}_t-X^{\Phi_m(n)}_s|) + 2\sum_{i=1}^kE(\sup_{0\leq s\leq T}|X_s||h^{\Phi_m(n)}_i-h_i|)\\
&+ 4kE(\sup_{0\leq s\leq T}|X_s|1_{\{|\tau_m-\tau_m^{\Phi_m(n)}|>\eta\}}) + 10 k E\big(\sup_{0\leq u\leq T} |X^{\Phi_m(n)}_u-X_u|\big)
\end{align*}
As in the proof of Part (i), successive extractions allow us to find $\lambda_m(n)$ (independent from $i$) such that for all $1\leq i\leq k$, $h^{\lambda_m(n)}_i$ converges a.s.~to $h_i$. Letting $n$ go to infinity in
\begin{align*}
|E(H\cdot A)_{\tau_m}|\leq K + &2 E(\sup_{s\leq t\leq s+\eta}|X^{\lambda_m(n)}_t-X^{\lambda_m(n)}_s|) + 2\sum_{i=1}^kE(\sup_{0\leq s\leq T}|X_s||h^{\lambda_m(n)}_i-h_i|)\\
&+ 4kE(\sup_{0\leq s\leq T}|X_s|1_{\{|\tau_m-\tau_m^{\lambda_m(n)}|>\eta\}}) + 10 k E\big(\sup_{0\leq u\leq T} |X^{\lambda_m(n)}_u-X_u|\big)
\end{align*}
gives the estimate
$$
|E(H\cdot A)_{\tau_m}|\leq K + 2 \limsup_{n\rightarrow\infty}E(\sup_{s\leq t\leq s+\eta}|X^n_t-X^n_s|)
$$
Let $\lim_{\eta\rightarrow 0}\limsup_{n\rightarrow\infty}E(\sup_{s\leq t\leq s+\eta}|X^n_s-X^n_t|)=C$. Since $E(\sup_{s\leq t\leq s+\eta}|X^n_t-X^n_s|)\leq 2E(\sup_u|X^n_u|)\leq 2(E(\sup_u|M^n_u|+\int_0^T|dA^n_s|))\leq 4K$, $C<\infty$. Now letting $\eta$ go to zero yields finally $|E(H\cdot A)_{\tau_m}|\leq K+2C$, for each $m$. Thus $E(\int_0^{\tau_m}|dA_s|)\leq K+2C$, for each~$m$ and hence $E(\int_0^T|dA_s|)\leq K+2C$.

Now, $M=X-A=(X-X^n)+M^n+A^n-A$, and so
$$
\sup_{0\leq s\leq T}|M_s|\leq \sup_{0\leq s\leq T}|X_s-X^n_s|+\sup_{0\leq s\leq T}|M^n|+\int_0^T|dA^n_s|+\int_0^T|dA_s|
$$
Thus $E(\sup_{0\leq s\leq T}|M_s|)\leq 3K + 2C$ and $M$ is a $\mathbb G$~martingale.
\end{proof}

Once one obtains that $X$ is a $\mathbb G$~special semimartingale, one can be interested in characterizing the martingale $M$ and the finite variation predictable process $A$ in terms of the processes $M^n$ and $A^n$. M\'emin (Theorem 11 in \cite{Memin}) achieved this under ``extended convergence.'' Recall that $(X^n, \mathbb G^n)$ converges to $(X, \mathbb G)$ in the extended sense if for every $G\in \mathbb G_T$, the sequence of c\`adl\`ag processes $(X^n_t, E(1_G\mid\mathcal G^n_t))_{0\leq t\leq T}$ converges in probability under the Skorohod $J_1$ topology to $(X_t, E(1_G\mid\mathcal G_t))_{0\leq t\leq T}$. The author proves the following theorem. We refer to \cite{Memin} for a proof. In the theorem below, $\mathbb G^n$ and $\mathbb G$ are right-continuous filtrations.

\begin{theorem}\label{T:memin}
Let $(X^n)_{n\geq 1}$ be a sequence of $\mathbb G^n$~special semimartingales with canonical decompositions $X^n=M^n+A^n$ where $M^n$ is a $\mathbb G^n$~martingale and $A^n$ is a $\mathbb G^n$~predictable finite variation process. We suppose that the sequence $([X^n, X^n]_T^{\frac{1}{2}})_{n\geq 1}$ is uniformly integrable and that the sequence $(V(A^n)_T)_{n\geq 1}$ (where $V$ denotes the variation process) of real random variables is tight in $\mathbb R$. Let $X$ be a $\mathbb G$~quasi-left continuous special semimartingale with a canonical decomposition $X=M+A$ such that $([X, X]_T^{\frac{1}{2}})<\infty$.

If the extended convergence $(X^n, \mathbb G^n) \rightarrow (X, \mathbb G)$ holds, then $(X^n, M^n, A^n)$ converges in probability under the Skorohod $J_1$ topology to $(X, M, A)$.
\end{theorem}

In a filtration expansion setting, the sequence $X^n$ is constant and equal to some semimartingale $X$ of the base filtration. In this case the extended convergence assumption in Theorem \ref{T:memin} reduces to the weak convergence of the filtrations. We can deduce the following corollary from Theorems \ref{T:semimg2} and \ref{T:memin}.

\begin{corollary}\label{cor:semimg}
Let $(\mathbb G^n)_{n\geq 1}$ be a sequence of right-continuous filtrations and let $\mathbb G$ be a filtration such that $\mathcal G^n_t\stackrel{w}{\rightarrow}\mathcal G_t$ for all $t$. Let $X$ be a stochastic process such that for each $n$, $X$ is a $\mathbb G^n$~semimartingale with canonical decomposition $X = M^n + A^n$ such that there exists $K>0$, $E(\int_0^T|dA^n_s|) \leq K$ and $E(\sup_{0\leq s\leq T}|M^n_s|)\leq K$ for all $n$. Then
\begin{itemize}
	\item[(i)] If $X$ is $\mathbb G$~adapted, then $X$ is a $\mathbb G$~special semimartingale. 

	\item[(ii)] Assume moreover that $\mathbb G$ is right-continuous and let $X=M+A$ be the canonical decomposition of $X$. Then $M$ is a $\mathbb G$ martingale and $\sup_{0\leq s\leq T} |M_s|$ and $\int_0^T|dA_s|$ are integrable. 

	\item[(iii)] Furthermore, assume that $X$ is $\mathbb G$~quasi-left continuous and $\mathbb G^n \stackrel{w}{\rightarrow} \mathbb G$. Then $(M^n, A^n)$ converges in probability under the Skorohod $J_1$ topology to $(M,A)$.
\end{itemize}
\end{corollary}

\begin{proof}
The sequence $X^n=X$ clearly satisfies the assumptions of Theorem \ref{T:semimg2}, and the two first claims follow. For the last claim, notice that $[X,X]_T \in L^1$, so $\sqrt{[X,X]_T}\in L^1$ and hence $(\sqrt{[X,X]_t})_{0\leq t\leq T}$ is a uniformly integrable family of random variables. The tightness of the sequence of random variables $(V(A^n)_T)_{n\geq 1}$ follows from $E(\int_0^T|dA^n_s|) \leq K$ for any $n$ and some $K$ independent from~$n$.
\end{proof}

\subsection{Applications to filtration expansions}

We provide in this subsection a first application to the initial and progressive filtration expansions with a random variable and a general theorem on the progressive expansion with a process. We assume in the sequel that a right-continuous filtration $\mathbb F$ is given.

\subsubsection{Initial and progressive filtration expansions with a random variable}

Assume that $\mathbb F$ is the natural filtration of some c\`adl\`ag process. Let $\tau$ be a random variable and $\mathbb H$ and $\mathbb G$ the initial and progressive expansions of $\mathbb F$ with~$\tau$. In this subsection, the filtration $\mathbb G$ is considered only when $\tau$ is non negative. It is proved in \cite{KLP} that if $\tau$ satisfies Jacod's criterion i.e. if there exists a $\sigma$-finite measure $\eta$ on $\mathcal{B}(\mathbb R)$ such that 
$$
P(\tau \in \cdot\mid\mathcal F_t)(\omega) \ll \eta(.) \quad \text{a.s.}
$$
then every $\mathbb F$~semimartingale remains an $\mathbb H$ and $\mathbb G$~semimartingale. That it is an $\mathbb H$~semimartingale is due to Jacod \cite{Jacod:1987}. That it is also a $\mathbb G$~semimartingale  follows from Stricker's theorem. Its $\mathbb G$~decomposition is obtained in \cite{KLP} and this relies on the fact that these two filtrations \textit{coincide} after $\tau$. We provide now a similar but partial result for a random variable $\tau$ which may not satisfy Jacod's criterion. Assume there exists a sequence of random times $(\tau_n)_{n\geq 0}$ converging in probability to $\tau$ and let $\mathbb H^n$ be the initial expansion of $\mathbb F$ with $\tau_n$. The following holds.

\begin{theorem}\label{T:ini}
Let $M$ be an $\mathbb F$~martingale such that $\sup_{0 \leq t\leq T}|M_t|$ is integrable. Assume there exists an $\mathbb H^n$ predictable finite variation process $A^n$ such that $M-A^n$ is an $\mathbb H^n$ martingale. If there exists $K$ such that $E(\int_0^T|dA^n_s|)\leq K$ for all $n$, then $M$ is an $\mathbb H$ and $\mathbb G$ semimartingale. 
\end{theorem}

\begin{proof}
Since $\tau_n$ converges in probability to $\tau$ and $\mathbb F$ is the natural filtration of some c\`adl\`ag process, we can prove that $\mathcal H^n_t \stackrel{w}{\rightarrow} \mathcal H_t$ for each $t\in [0,T]$, using the same techniques as in Lemmas \ref{L:carac} and \ref{L:conv}. Up to replacing $K$ by $K+E(\sup_{0 \leq t\leq T}|M_t|)$, $M^n=M-A^n$ and $A^n$ satisfy the assumptions of Corollary~\ref{cor:semimg}. Therefore $M$ is an $\mathbb H$~semimartingale, and a $\mathbb G$~semimartingale by Stricker's theorem. 
\end{proof}

One case where the first assumption of Theorem \ref{T:ini} is satisfied is when $\tau_n$ satisfies Jacod's criterion, for each $n\geq 0$. In this case, and if $\mathbb F$ is the natural filtration of a Brownian motion $W$, the result above can be made more explicit. Assume for simplicity that the conditional distributions of $\tau_n$ are absolutely continuous w.r.t~Lebesgue measure,
$$
P(\tau_n\in du\mid\mathcal F_t)(\omega)=p^n_t(u,\omega)du
$$
where the conditional densities are chosen so that $(u,\omega, t)\rightarrow p^n_t(u,\omega)$ is c\`adl\`ag in $t$ and measurable for the optional $\sigma$-field associated with the filtration $\hat{\mathbb F}$ given by $\hat{\mathcal F}_t=\cap_{u>t}\mathcal B(\mathbb R)\otimes \mathcal F_u$. From the martingale representation theorem in a Brownian filtration, there exists for each $n$ a family $\{q^n(u), u>0\}$ of $\mathbb F$~predictable processes $(q^n_t(u))_{0\leq t\leq T}$ such that
\begin{equation}\label{eq:dens}
p^n_t(u)=p^n_0(u)+\int_0^tq^n_s(u)dW_s
\end{equation}

\begin{corollary}\label{cor:ini}
Assume there exists $K$ such that $E\Big(\int_0^T\Big|\frac{q^n_s(\tau_n)}{p^n_s(\tau_n)}\Big|ds\Big)\leq K$ for all $n$, then $W$ is a special semimartingale in both $\mathbb H$ and $\mathbb G$.
\end{corollary}

\begin{proof}
Since $\tau_n$ satisfy Jacod's criterion, it follows from Theorem 2.1 in \cite{Jacod:1987} that $W_t-A^n_t$ is an $\mathbb H$~local martingale, where $A^n_t=\int_0^t\frac{d\langle p^n(u), W\rangle_s}{p^n_s(u)}\Big|_{u=\tau_n}$. Now, it follows from (\ref{eq:dens}) that
$A^n_t=\int_0^t\frac{q^n_s(\tau_n)}{p^n_s(\tau_n)}ds$ and Theorem \ref{T:ini} allows us to conclude. 
\end{proof}

Assume the assumptions of Corollary \ref{cor:ini} are satisfied and let $W=M+A$ be the $\mathbb H$~canonical decomposition of $W$. Let $m$ be an $\mathbb F$~predictable process such that $\int_0^t m_s^2ds$ is locally integrable and let $M$ be the $\mathbb F$~local martingale $M_t=\int_0^tm_sdW_s$. Theorem VI.5 in \cite{Protter:2005} then guarantees that $M$ is an $\mathbb H$~semimartingale as soon as the process $(\int_0^tm_sdA_s)_{t\geq 0}$ exists as a path-by-path Lebesgue-Stieltjes integral a.s.

\begin{example}
In order to emphasize that some assumptions as in Theorem \ref{T:ini} are needed, we provide now a counter-example. Let $\mathbb F$ be the natural filtration of some Brownian motion~$B$ and choose $\tau$ to be some functional of the Brownian path i.e.~$\tau=f((B_s, 0\leq s\leq 1))$, such that $\sigma(\tau)=\mathcal F_1$. Then $B$ is not a semimartingale in $\mathbb H=(\mathcal F_t\vee\sigma(\tau))_{0\leq t\leq 1}$. Now, define $\tau_n=\tau+\frac{1}{\sqrt{n}}N$, where $N$ is a standard normal random variable independent from $\mathbb F$. Then $\tau_n$ converge a.s.~to $\tau$ and $P(\tau_n\leq u\mid\mathcal F_t)=\int_{-\infty}^uE(g_n(v-\tau)\mid\mathcal F_t)dv$ where $g_n$ is the probability density function of $\frac{1}{\sqrt{n}}N$, hence $P(\tau_n\in du\mid\mathcal F_t)(\omega)=E(g_n(u-\tau)\mid\mathcal F_t)(\omega)(du)$. Therefore, $\tau_n$ satisfies Jacod's criterion, for each $n$ and $p^n_t(u,\omega)=E(g_n(u-\tau)\mid\mathcal F_t)(\omega)$. Thus, $B$ is a semimartingale in $\mathbb H^n=(\mathcal F_t\vee\sigma(\tau_n)_{0\leq t\leq 1})$ and $\mathcal H^n_t\stackrel{w}{\rightarrow} \mathcal H_t$ for each $0\leq t\leq 1$.
\end{example}

\subsubsection{Progressive filtration expansion with a process}

Let $(N^n)_{n\geq 1}$ be a sequence of c\`adl\`ag processes converging in probability under the Skorohod $J_1$-topology to a c\`adl\`ag process $N$ and let $\mathbb{N}^n$ and $\mathbb{N}$ be their natural filtrations. Define the filtrations $\mathbb G^{0,n}=\mathbb{F}\vee\mathbb{N}^n$ and $\mathbb G^n$ by $\mathcal G_t^n=\bigcap_{u>t}\mathcal G^{0,n}_u$. Let also $\mathbb G^0$ (resp. $\mathbb G$) be the smallest (resp. the smallest right-continuous) filtration containing $\mathbb F$ and to which $N$ is adapted. The result below is the main theorem of this paper.

\begin{theorem}\label{T:exp}
Let $X$ be an $\mathbb{F}$~semimartingale such that for each $n$, $X$ is a $\mathbb{G}^n$~semimartingale with canonical decomposition $X = M^n + A^n$. Assume $E(\int_0^T|dA^n_s|) \leq K$ and $E(\sup_{0\leq s\leq T}|M^n_s|)\leq K$ for some $K$ and all $n$. Finally, assume one of the following holds.
\begin{itemize}
	\item[-] $N$ has no fixed times of discontinuity,
	
	\item[-] $N^n$ is a discretization of $N$ along some refining subdivision $(\pi_n)_{n\geq 1}$ such that each fixed time of discontinuity of $N$ belongs to $\cup_n \pi_n$.
\end{itemize}
Then 
\begin{itemize}
	\item[(i)] $X$ is a $\mathbb G^0$~special semimartingale. 

	\item[(ii)] Moreover, if $\mathbb F$ is the natural filtration of some c\`adl\`ag process then $X$ is a $\mathbb G$~special semimartingale with canonical decomposition $X=M+A$ such that $M$ is a $\mathbb{G}$ martingale and $\sup_{0\leq s\leq T} |M_s|$ and $\int_0^T|dA_s|$ are integrable.

	\item[(iii)] Furthermore, assume that $X$ is $\mathbb G$~quasi-left continuous and $\mathbb G^n \stackrel{w}{\rightarrow} \mathbb G$. Then $(M^n, A^n)$ converges in probability under the Skorohod $J_1$ topology to $(M,A)$.
\end{itemize}
\end{theorem}

\begin{proof}
Under assumption (i), since $N^n\stackrel{P}{\rightarrow}N$ and $P(\Delta N_t \neq 0)=0$ for all $t$, it follows from Lemma \ref{L:conv} that $\mathcal N^n_t\stackrel{P}{\rightarrow}\mathcal N_t$ for all $t$. The same holds under assumption (ii) using Lemma \ref{L:discr}. Lemma \ref{L:vee} then ensures that $\mathcal G^{0,n}_t\stackrel{w}{\rightarrow}\mathcal G^0_t$ for all $t$. Since $\mathcal G^{0,n}_t\subset\bigcap_{u>t}\mathcal G^{0,n}_u=\mathcal G^n_t$, it follows from Lemma \ref{L:subset} that $\mathcal G^n_t\stackrel{w}{\rightarrow}\mathcal G^0_t$ for all $t$. Being an $\mathbb{F}$~semimartingale, $X$ is clearly $\mathbb G^0$~adapted. An application of Corollary \ref{cor:semimg} ends the proof of the first claim. When $\mathbb F$ is the natural filtration of some c\`adl\`ag process, the same proofs as of Lemmas \ref{L:conv} and \ref{L:discr} guarantee that $\mathcal G^n_t\stackrel{w}{\rightarrow}\mathcal G_t$ for all $t$. Since $\mathbb G$ is right-continuous, the second and third claims follow from Corollary \ref{cor:semimg}.
\end{proof}

We apply this result to expand the filtration $\mathbb F$ progressively with a point process. Let $(\tau_i)_{i\geq 1}$ and $(X_i)_{i\geq 1}$ be two sequences of random variables such that for each $n$, the random vector $(\tau_1, X_1, \ldots, \tau_n, X_n)$ satisfies Jacod's criterion w.r.t the filtration $\mathbb F$. Assume that for all $t$ and $i$, $P(\tau_i=t)=0$ and that one of the following holds:
\begin{itemize}
	\item[(i)] For all $i$, $X_i$ and $\tau_i$ are independent, $E|X_i|=\mu$ for some $\mu$ and $\sum_{i=1}^{\infty}P(\tau_i\leq T)<\infty$
	
	\item[(ii)] $E(|X_i^2|)=c$ and $\sum_{i=1}^{\infty}\sqrt{P(\tau_i\leq T)}< \infty$.
\end{itemize}
Let $N^n_t=\sum_{i=1}^nX_i1_{\{\tau_i\leq t\}}$ and $N_t=\sum_{i=1}^{\infty}X_i1_{\{\tau_i\leq t\}}$. The assumptions on $N^n$ and $N$ as of Theorem \ref{T:exp} are satisfied.
\begin{lemma}
Under the assumptions above, $N_t\in L^1$ for each $t$, $N^n\stackrel{P}{\rightarrow} N$ and $N$ has no fixed times of discontinuity. 
\end{lemma}
\begin{proof}
We prove the statement under assumption (i). For each~$t$,
$$
E(|N_t|)\leq \sum_{i=1}^{\infty}E(|X_i|1_{\{\tau_i\leq t\}})\leq \mu\sum_{i=1}^{\infty}P(\tau_i\leq t)<\infty
$$
Therefore, $N_t\in L^1$. For $\eta>0$ and $n$ integer, we obtain the following estimate.
\begin{align*}
P(\sup_{0\leq t\leq T}|N_t-N^n_t|\geq\eta)&=P(\sup_{0\leq t\leq T}\big|\sum_{i=n+1}^{\infty}X_i1_{\{\tau_i\leq t\}}\big|\geq \eta)\\
&\leq P(\sup_{0\leq t\leq T}\sum_{i=n+1}^{\infty}|X_i|1_{\{\tau_i\leq t\}}\geq \eta)=P(\sum_{i=n+1}^{\infty}|X_i|1_{\{\tau_i\leq T\}}\geq \eta)\\
&\leq\frac{1}{\eta}E(\sum_{i=n+1}^{\infty}|X_i|1_{\{\tau_i\leq T\}})=\frac{\mu}{\eta}\sum_{i=n+1}^{\infty}P(\tau_i\leq T)\rightarrow 0
\end{align*}
This implies $N^n\stackrel{P}{\rightarrow} N$. Under assumption (ii), the proof is also straightforward and based on Cauchy Schwarz inequalities. Finally, since
$$
P(|\Delta N_t|\neq 0)\leq P(\exists i \mid \tau_i=t)\leq \sum_{i=1}^{\infty}P(\tau_i=t)=0
$$
$N$ has no fixed times of discontinuity
\end{proof}

Since the random vector $(\tau_1, X_1, \ldots, \tau_n, X_n)$ is assumed to satisfy Jacod's criterion, it follows from \cite{KLP} that $\mathbb F$~semimartingales remain $\mathbb G^n$~semimartingales, for each $n$. Therefore, this property also holds between $\mathbb F$ and $\mathbb G$ for $\mathbb F$~semimartingales whose $\mathbb G^n$~canonical decompositions satisfy the regularity assumptions of Theorem~\ref{T:exp}. Here $\mathbb G$ is the smallest filtration containing $\mathbb F$ and to which $N$ is adapted.

We would like to take a step further and reverse the previous situation. That is instead of starting with a sequence of processes $N^n$ converging to some process $N$, and putting assumptions on the semimartingale properties of $\mathbb F$~semimartingales w.r.t the intermediate filtrations~$\mathbb G^n$ and their decompositions therein, we would like to expand the filtration $\mathbb F$ with a given process $X$ and express all the assumptions in terms of $X$ and the $\mathbb F$~semimartingales considered. We are able to do this for c\`adl\`ag processes which satisfy a criterion that can loosely be seen as a localized extension of Jacod's criterion to processes. The integrability assumptions of Theorem \ref{T:exp} are expressed in terms of $\mathcal F_t$-conditional densities. Before doing this, we conclude this section by studying the stability of hypothesis $(H)$ with respect to the weak convergence of the $\sigma$-fields in a filtration expansion setting.

\subsection{The case of hypothesis $(H)$}
Recall that given two nested filtrations $\mathbb F\subset\mathbb G$, we say that hypothesis $(H)$ holds between $\mathbb F$ and $\mathbb G$ if any square integrable $\mathbb F$ martingale remains an $\mathbb G$ martingale. Br\'emaud and Yor proved the next Lemma~\ref{L:H} (see \cite{BremaudYor}).
\begin{lemma}\label{L:H}
Let $\mathbb F\subset\mathbb G$ two nested filtrations. The following assertions are equivalent.
\begin{itemize}
\item[(i)] Hypothesis $(H)$ holds between $\mathbb F$ and $\mathbb G$.
 
\item[(ii)] For each $0\leq t\leq T$, $\mathcal F_T$ and $\mathcal G_t$ are conditionally independent given $\mathcal F_t$.

\item[(iii)] For each $0\leq t\leq T$, each $F\in L^2(\mathcal F_T)$ and each $G_t\in L^2(\mathcal G_t)$, 
$$
E(FG_t\mid\mathcal F_t)=E(F\mid\mathcal F_t)E(G_t\mid\mathcal F_t).
$$
\end{itemize}
\end{lemma}

Let $\mathbb F\subset\mathbb G$ be two nested right-continuous filtrations and $\mathbb G^n$ be a sequence of right-continuous filtrations containing $\mathbb F$ and such that $\mathcal G^n_t$ converges weakly to $\mathcal G_t$ for each $t$. We mentioned that an $\mathbb F$~local martingale that remains a $\mathbb G^n$~semimartingale for each $n$ might still lose its semimartingale property in $\mathbb G$ and we provided conditions that prevent this pathological behavior. In this subsection, we prove that this cannot happen in case hypothesis $(H)$ holds between $\mathbb F$ and each $\mathbb G^n$. One obtains even that hypothesis $(H)$ holds between $\mathbb F$ and $\mathbb G$.

\begin{theorem}
Let $\mathbb F$, $\mathbb G$ and $(\mathbb G^n)_{n\geq 1}$ right-continuous filtrations such that $\mathbb F\subset\mathbb G$, $\mathbb F\subset\mathbb G^n$ for each $n$ and $\mathcal G^n_t\stackrel{w}{\rightarrow}\mathcal G_t$ for each $t$. Assume that for each $n$, hypothesis $(H)$ holds between $\mathbb F$ and $\mathbb G^n$. Then hypothesis $(H)$ holds between $\mathbb F$ and $\mathbb G$.
\end{theorem}
\begin{proof}
We use Lemma \ref{L:H} and start with the bounded case. Let $0\leq t\leq T$, $F\in L^2(\mathcal F_T)$ and $G_t\in L^{\infty}(\mathcal G_t)$. For each $n$, define $G^n_t=E(G_t\mid\mathcal G^n_t)$. Then $G^n_t\in L^{\infty}(\mathcal G^n_t)$. Since hypothesis $(H)$ holds between $\mathbb F$ and $\mathbb G^n$, Lemma \ref{L:H} guarantees that $E(FG^n_t\mid\mathcal F_t)=E(F\mid\mathcal F_t)E(G^n_t\mid\mathcal F_t)$. But $\mathcal F_t\subset\mathcal G^n_t$, hence $E(G^n_t\mid\mathcal F_t)=E(E(G_t\mid\mathcal G^n_t)\mid\mathcal F_t)=E(G_t\mid\mathcal F_t)$. Since $\mathcal G^n_t\stackrel{w}{\rightarrow}\mathcal G_t$, $FG^n_t\stackrel{P}{\rightarrow}FG_t$. Now $FG^n_t$ is bounded by a square integrable process (by assumption) so the convergence holds in $L^1$ by the Dominated Convergence theorem so that $E(FG^n_t\mid\mathcal F_t)\stackrel{P}{\rightarrow}E(FG_t\mid\mathcal F_t)$. This proves that $E(FG_t\mid\mathcal F_t)=E(F\mid\mathcal F_t)E(G_t\mid\mathcal F_t)$. The general case where $G_t\in L^2(\mathcal G_t)$ follows by applying the bounded case result to the bounded random variables $G^{(m)}_t=G_t\wedge m$. Then for each $m$,
$$
E(FG^{(m)}_t\mid\mathcal F_t)\stackrel{P}{\rightarrow}E(FG^{(m)}_t\mid\mathcal F_t)
$$
and the Monotone Convergence theorem allows us to conclude.
\end{proof}

\section{Filtration expansion with a c\`adl\`ag process satisfying a \textit{generalized Jacod's criterion} and applications to diffusions}

In this section, we assume a c\`adl\`ag process $X$ and a right-continuous filtration $\mathbb F$ are given. We study the case where the process $X$ and the filtration $\mathbb F$ satisfy the following assumption.

\begin{assumption}[Generalized Jacod's criterion]\label{A:J}
There exists a sequence $(\pi_n)_{n\geq 1}=(\{t^n_i\})_{n\geq 1}$ of subdivisions of $[0,T]$ whose mesh tends to zero and such that for each $n$, $(X_{t^n_0}, X_{t^n_1}-X_{t^n_0}, \ldots, X_{T}-X_{t^n_n})$ satisfies Jacod's criterion, i.e.~there exists a $\sigma$-finite measure $\eta_n$ on $\mathcal{B}(\mathbb{R}^{n+2})$ such that $P\big((X_{t^n_0}, X_{t^n_1}-X_{t^n_0}, \ldots, X_{T}-X_{t^n_n})\in \cdot\mid\mathcal F_t\big)(\omega) \ll \eta_n(\cdot)$~a.s.
\end{assumption}
Under Assumption \ref{A:J}, the $\mathcal F_t$-conditional density
$$
p^{(n)}_t(u_0,\ldots,u_{n+1},\omega)=\frac{P\big((X_{t^n_0}, X_{t^n_1}-X_{t^n_0}, \ldots, X_T-X_{t^n_n})\in (du_0,\ldots,du_{n+1})\mid\mathcal F_t\big)(\omega)}{\eta_n(du_0,\ldots,du_{n+1})}
$$
exists for each $n$, and can be chosen so that $(u_0,\ldots,u_{n+1}, \omega, t)\rightarrow p^{(n)}_t(u_0,\ldots,u_{n+1},\omega)$ is c\`adl\`ag in $t$ and measurable for the optional $\sigma$-field associated with the filtration $\hat{\mathbb{F}}_t$ given by $\hat{\mathcal{F}}_t=\cap_{u>t}\mathcal{B}(\mathbb{R}^{n+2})\otimes \mathcal{F}_u$.
For each $0\leq i\leq n$, define 
$$
p^{i,n}_t(u_0,\ldots,u_i)=\int_{\mathbb{R}^{n+1-i}}p^{(n)}_t(u_0,\ldots,u_{n+1})\eta_n(du_{i+1},\ldots,du_{n+1})
$$
Let $M$ be a continuous $\mathbb F$~local martingale. Define
\begin{equation}\label{eq:Ain}
A^{i,n}_t=\int_{0}^t\frac{d\langle p^{i,n}(u_0,\ldots,u_i),M\rangle_s}{p^{i,n}_{s^{-}}(u_0,\ldots,u_i)}\Big|_{\forall 0\leq k\leq i, u_k=X_{t^n_k}-X_{t^n_{k-1}}}
\end{equation}
Finally define
$$
A^{(n)}_t=\sum_{i=0}^n\int_{t\wedge t^n_i}^{t\wedge t^n_{i+1}}dA^{i,n}_s
$$
i.e.
\begin{equation}\label{eq:An}
A^{(n)}_t=\sum_{i=0}^n1_{\{t^n_i\leq t<t^n_{i+1}\}}\big(\sum_{k=0}^{i-1}\int_{t^n_k}^{t^n_{k+1}}dA^{k,n}_s+\int_{t^n_i}^tdA^{i,n}_s\big)
\end{equation}
Of course, on each time interval $\{t^n_i\leq t<t^n_{i+1}\}$, only one term appears in the outer sum. Let $\mathbb G^0$ (resp. $\mathbb G$) be the smallest (resp. the smallest right-continuous) filtration containing $\mathbb F$ and relative to which $X$ is adapted. The theorem below is the main result of this section.

\begin{theorem}\label{T:semimgX}
Assume $X$ and $\mathbb F$ satisfy Assumption \ref{A:J} and that one of the following holds.
\begin{itemize}
	\item[-] $X$ has no fixed times of discontinuity,
	
	\item[-] the sequence of subdivisions $(\pi_n)_{n\geq 1}$ in Assumption \ref{A:J} is refining and each fixed time of discontinuity of $X$ belongs to $\cup_n \pi_n$.
\end{itemize}
Let $M$ be a continuous $\mathbb F$~martingale such that $E(\sup_{s\leq T}|M_s|)\leq K$ and $E(\int_0^T|dA^{(n)}_s|)\leq K$ for some $K$ and all $n$, with $A^n$ as in (\ref{eq:An}). Then
\begin{itemize}
	\item[(i)] $M$ is a $\mathbb G^0$~special semimartingale. 

	\item[(ii)] Moreover, if $\mathbb F$ is the natural filtration of some c\`adl\`ag process $Z$, then $M$ is a $\mathbb G$~special semimartingale with canonical decomposition $M=N+A$ such that $N$ is a $\mathbb G$ martingale and $\sup_{0\leq s\leq T} |N_s|$ and $\int_0^T|dA_s|$ are integrable.
\end{itemize}
\end{theorem}

\begin{proof}
We construct the discretized process $X^n$ defined by $X^n_t=X_{t^n_k}$ for all $t^n_k\leq t<t^n_{k+1}$. That is 
$$
X^n_t=\sum_{i=0}^nX_{t^n_i}1_{\{t^n_i\leq t<t^n_{i+1}\}} + X_T1_{\{t=T\}}
$$
with the convention $t^n_0=0$ and $t^n_{n+1}=T$. Let $\mathbb G^n$ be the smallest right-continuous filtration containing $\mathbb F$ and to which $X^n$ is adapted.

Now, for $0\leq t\leq T$,
\begin{align*}
X^n_t&=\sum_{i=0}^nX_{t^n_i}1_{\{t^n_i\leq t<t^n_{i+1}\}}+X_T1_{\{t=T\}}=\sum_{i=0}^nX_{t^n_i}1_{\{t^n_i\leq t\}}-\sum_{i=0}^nX_{t^n_i}1_{\{t^n_{i+1}\leq t\}}+X_T1_{\{t=T\}}\\
&=\sum_{i=1}^n(X_{t^n_i}-X_{t^n_{i-1}})1_{\{t^n_i\leq t\}}+X_{0}1_{\{t^n_0\leq t\}}-X_{t^n_{n}}1_{\{t^n_{n+1}\leq t\}}+X_T1_{\{t^n_{n+1}\leq t\}}\\
&=X_{0}1_{\{t^n_0\leq t\}}+\sum_{i=1}^{n+1}(X_{t^n_i}-X_{t^n_{i-1}})1_{\{t^n_i\leq t\}}=\sum_{i=0}^{n+1}(X_{t^n_i}-X_{t^n_{i-1}})1_{\{t^n_i\leq t\}}
\end{align*}
with the notation $X_{t^n_{-1}}=0$.

For each $0\leq i\leq n+1$, let $\mathbb H^{i,n}$ be the initial expansion of $\mathbb F$ with $(X_{t^n_k}-X_{t^n_{k-1}})_{0\leq k\leq i}$. Since $(X_{t^n_k}-X_{t^n_{k-1}})_{0\leq k\leq i}$ satisfies Jacod's criterion, it follows that for each $0\leq i\leq n+1$, $M-A^{i,n}$ is an $\mathbb H^{i,n}$~local martingale. Let 
$$
\tilde{\mathcal G}^n_t=\bigcap_{u>t}\mathcal{F}_u\vee\sigma\big((X_{t^n_i}-X_{t^n_{i-1}})1_{\{t^n_i\leq u\}}, i=0,\ldots,n+1\big)
$$
Since the times $t^n_k$ are fixed, $\mathbb H^{i,n}$ is also the initial expansion of $\mathbb F$ with $(t^n_k, X_{t^n_k}-X_{t^n_{k-1}})_{0\leq k\leq i}$ and $\tilde{\mathbb G}^n=\mathbb G^n$ using a Monotone Class argument and the fact that $X^n_{t^n_k}=X_{t^n_k}$, for all $0\leq k\leq n+1$. So it follows from Theorem 8 in \cite{KLP} that $M-A^{(n)}$ is a $\mathbb G^n$~local martingale. An application of Theorem \ref{T:exp} yields the result.
\end{proof}

We refrain from stating Theorem \ref{T:semimgX} in a more general form for clarity but provide two extensions in the remarks below.
\begin{itemize}

	\item[(i)] Going beyond the continuous case for the $\mathbb F$ local martingale $M$  is straightforward. We only need to use Theorem 8 in \cite{KLP} in its general version rather than its application to the continuous case. However the explicit form of $A^{(n)}$ is much more complicated, which makes it hard to check the integrability assumption of Theorem~\ref{T:semimgX}. To be more concrete, one has to replace $A^{(n)}$ in the theorem above by $\tilde{A}^{(n)}$ defined by
$$
\tilde{A}^{(n)}_t=\sum_{i=0}^n\int_{t\wedge t^n_i}^{t\wedge t^n_{i+1}}(d\tilde{A}^{i,n}_s+dJ^{i,n}_s)
$$
i.e.
$$
\tilde{A}^{(n)}_t=\sum_{i=0}^n1_{\{t^n_i\leq t<t^n_{i+1}\}}\Big(\sum_{k=0}^{i-1}\int_{t^n_k}^{t^n_{k+1}}(d\tilde{A}^{k,n}_s+dJ^{k,n}_s)+\int_{t^n_i}^t(d\tilde{A}^{i,n}_s+dJ^{i,n}_s)\Big)
$$
where $\tilde{A}^{i,n}$ is the compensator of $M$ in $\mathbb{H}^{i,n}$ as given by Jacod's theorem (see Theorems VI.10 and VI.11 in \cite{Protter:2005}) and $J^{i,n}$ is the dual predictable projection of $\Delta M_{t^n_{i+1}}1_{[t^n_{i+1},\infty[}$ onto~$\mathbb H^{i,n}$.

	\item[(ii)] A careful study of the proof above shows that Assumption \ref{A:J} is only used to ensure that there exists an $\mathbb H^{i,n}$ predictable process $A^{i,n}$ such that $M-A^{i,n}$ is an $\mathbb H^{i,n}$~local martingale. Therefore, Theorem \ref{T:semimgX} will hold whenever this weaker assumption is satisfied.
\end{itemize}

If the sequence of filtrations $\mathbb G^n$ converges weakly to $\mathbb G$ then $(M-A^{(n)}, A^{(n)})$ converges in probability under the Skorohod $J_1$ topology to $(N, A)$. Many criteria for this to hold are provided in the literature, see for instance Propositions $3$ and $4$ in \cite{Coquet0}. This holds for example when every $\mathbb G$~martingale is continuous and the subdivision $(\pi_n)_{n\geq 1}$ is refining. In this case, for each $0\leq t\leq T$, $(\mathcal G^n_t)_{n\geq 1}$ is increasing and converges weakly to the $\sigma$-field~$\mathcal G_t$. The following lemma allows us to conclude. See \cite{Coquet0} for a proof. 

\begin{lemma}\label{L:conv2}
Assume that every $\mathbb G$~martingale is continuous and that for every $0\leq t\leq T$, $(\mathcal G^n_t)_{n\geq 1}$ increases (or decreases) and converges weakly to $\mathcal G_t$. Then $\mathbb G^n\stackrel{w}{\rightarrow}\mathbb G$.
\end{lemma}


\subsection{Application to diffusions}

Start with a Brownian filtration $\mathbb F=(\mathcal F_t)_{0\leq t\leq T}$, $\mathcal F_t=\sigma(B_s, s\leq t)$ and consider the stochastic differential equation
$$
dX_t=\sigma(X_t)dB_t+b(X_t)dt
$$
Assume the existence of a unique strong solution $(X_t)_{0\leq t\leq T}$. Indeed, assume the transition density $\pi(t,x,y)$ exists and is twice continuously differentiable in $x$ and continuous in $t$ and $y$. This is guaranteed for example if $b$ and $\sigma$ are infinitely differentiable with bounded derivatives and if the H\"{o}rmander condition holds for any~$x$ (see \cite{BallyTalay}), and we assume that this holds in the sequel. In this case, $\pi$ is even infinitely differentiable. 

We next show how we can expand a filtration dynamically as $t$ increases, via another stochastic process evolving backwards in time.  To this end, define the time reversed process $Z_t=X_{T-t}$, for all $0\leq t\leq T$. Let $\mathbb G=(\mathcal G_t)_{0\leq t<\frac{T}{2}}$ be the smallest right-continuous filtration containing $(\mathcal F_t)_{0\leq t<\frac{T}{2}}$ and to which $(Z_t)_{0\leq t<\frac{T}{2}}$ is adapted.
 We would like to prove that $B$ remains a special semimartingale in $\mathbb G$ and give its canonical decomposition. That $B$ is a $\mathbb G$ semimartingale can be obtained using the usual results from the filtration expansion theory. However, our approach allows us to obtain the decomposition, too. We assume (w.l.o.g) that $T=1$. Introduce the reversed Brownian motion $\tilde{B}_t=B_{1-t}-B_1$ and the filtration $\tilde{\mathbb G}=(\tilde{\mathcal{G}_t})_{0\leq t<\frac{1}{2}}$ defined by
$$
\tilde{\mathcal G}_t=\bigcap_{t<u<\frac{1}{2}}\sigma(B_s, \tilde{B}_s, 0\leq s<u).
$$

\begin{theorem}
Both $B$ and $\tilde{B}$ are $\mathbb G$~semimartingales.
\end{theorem}

\begin{proof}
First, it is well known that $\tilde{B}$ is a Brownian motion in its own natural filtration and $\sigma(B_{1-s}-B_1, 0\leq s<\frac{1}{2})$ is independent from $\sigma(B_s, 0\leq s<\frac{1}{2})$. Therefore $(B_t)_{0\leq t<\frac{1}{2}}$ and $(\tilde{B}_t)_{0\leq t<\frac{1}{2}}$ are independent Brownian motions in $\tilde{\mathbb G}$. Now, given our strong assumptions on the coefficients $b$ and $\sigma$, $X_1$ satisfies Jacod's criterion with respect to $\tilde{\mathbb G}$. Therefore $B$ and $\tilde{B}$ remain semimartingales in $\mathbb H=(\mathcal H_t)_{0\leq t<\frac{1}{2}}$ where $\mathcal H_t=\bigcap_{\frac{1}{2}>u>t}\tilde{\mathcal G}_u\vee \sigma(X_1)$. It only remains to prove that $\mathbb G=\mathbb H$. For this, use Theorem~$V.23$ in \cite{Protter:2005} to get that 
$$
dX_{1-t}=\sigma(X_{1-t})d\tilde{B}_t+(\sigma^{'}(X_{1-t})\sigma(X_{1-t})+b(X_{1-t}))dt
$$
Since $b+\sigma\sigma^{'}$ and $\sigma$ are Lipschitz, $\bigcap_{\frac{1}{2}>u>t}\sigma(X_{1-s}, 0\leq s\leq u)=\bigcap_{\frac{1}{2}>u>t}\sigma(\tilde{B}_s, 0\leq s\leq u)\vee \sigma(X_1)$ and the result follows.
\end{proof}

We apply now our results to obtain the $\mathbb G$~decomposition. 
This is the primary result of this article.

\begin{theorem}\label{T:diff}
Assume there exists a nonnegative function $\phi$ such that $\int_0^1\phi(s)ds <\infty$ and for each $0\leq s < t$,
$$
E\Big(\Big|\frac{1}{\pi}\frac{\partial \pi}{\partial x}(t-s, X_s, X_t)\Big|\Big)\leq \phi(t-s)
$$
Then the process $(B_t)_{0\leq t<\frac{1}{2}}$ is a $\mathbb G$ semimartingale and 
$$
B_t - \int_0^t\frac{1}{\pi}\frac{\partial \pi}{\partial x}(1-2s, X_s, X_{1-s})ds
$$
is a $\mathbb G$~Brownian motion.
\end{theorem}

\begin{proof}
Since the process $Z_t$ is a c\`adl\`ag process with no fixed times of discontinuity, we can apply Theorem \ref{T:semimgX}. First we prove that $(Z_t)_{0\leq t<\frac{1}{2}}$ and $(\mathcal F_t)_{0\leq t<\frac{1}{2}}$ satisfy Assumption~\ref{A:J}. Let $(\pi_n)_{n\geq 1}=(\{t^n_i\})_{n\geq 1}$ be a refining  sequence of subdivisions of $[0,\frac{1}{2}]$ whose mesh tends to zero. We will do more and compute directly the conditional distributions of $(Z_{t^n_0}, Z_{t^n_1}-Z_{t^n_0}, \ldots, Z_{t^n_i}-Z_{t^n_{i-1}})$ for any $1\leq i\leq n+1$. Pick such $i$ and let $0\leq t<\frac{1}{2}$ and $(z_0,\ldots, z_i)\in~\mathbb R^{i+1}$.
\begin{align*}
P(Z_{t^n_0}&\leq z_0, Z_{t^n_1}-Z_{t^n_0}\leq z_1, \ldots, Z_{t^n_i}-Z_{t^n_{i-1}}\leq z_i\mid\mathcal F_t)\\
&=P(X_1\leq z_0, X_1-X_{1-t^n_1}> -z_1, \ldots, X_{1-t^n_{i-1}}-X_{1-t^n_i}> -z_i\mid\mathcal F_t)\\
&=E\Big(\prod_{k=1}^{i-1}1_{\{X_{1-t^n_k}-X_{1-t^n_{k+1}}\geq z_{k+1}\}} P\big(X_{1-t^n_1}-z_1\leq X_1\leq z_0\mid \mathcal F_{1-t^n_1}\big)\mid \mathcal F_t\Big)\\
&=E\Big(\prod_{k=1}^{i-1}1_{\{X_{1-t^n_k}-X_{1-t^n_{k+1}}\geq z_{k+1}\}} \int_{X_{1-t^n_1}-z_1}^{\infty} 1_{\{u_1\leq z_0\}}P_{X_{1-t^n_1}}(t^n_1,u_1)du_1\mid \mathcal F_t\Big)\\
&=E\Big(\prod_{k=1}^{i-1}1_{\{X_{1-t^n_k}-X_{1-t^n_{k+1}}\geq z_{k+1}\}} \int_{-z_1}^{\infty} 1_{\{v_1\leq z_0-X_{1-t^n_1}\}}P_{X_{1-t^n_1}}(t^n_1,v_1+X_{1-t^n_1})dv_1\mid \mathcal F_t\Big)
\end{align*}
Repeating the same technique and conditioning successively w.r.t $\mathcal F_{1-t^n_2}, \ldots, \mathcal F_{1-t^n_i}$ gives
\begin{align*}
P(Z_{t^n_0}&\leq z_0, Z_{t^n_1}-Z_{t^n_0}\leq z_1, \ldots, Z_{t^n_i}-Z_{t^n_{i-1}}\leq z_i\mid\mathcal F_t)=E\Big(\int_{-z_i}^{\infty}\cdots\int_{-z_1}^{\infty}\\
&1_{\{\sum_{k=1}^iv_k\leq z_0-X_{1-t^n_i}\}}\prod_{k=1}^iP_{X_{1-t^n_i}+\sum_{j=k+1}^iv_j}(t^n_k-t^n_{k-1},\sum_{l=k}^iv_l+X_{1-t^n_i})dv_1\ldots dv_i\mid \mathcal F_t\Big)\\
&=\int_{-\infty}^{\infty}\int_{-z_i}^{\infty}\cdots\int_{-z_1}^{\infty}1_{\{u+\sum_{k=1}^iv_k\leq z_0\}}P_{X_t}(1-t^n_i-t,u)\\
&\qquad \qquad \qquad \qquad\prod_{k=1}^iP_{u+\sum_{j=k+1}^iv_j}(t^n_k-t^n_{k-1},u+\sum_{l=k}^iv_l)dv_1\ldots dv_idu
\end{align*}
Fubini's Theorem implies then
\begin{align*}
P(Z_{t^n_0}&\leq z_0, Z_{t^n_1}-Z_{t^n_0}\leq z_1, \ldots, Z_{t^n_i}-Z_{t^n_{i-1}}\leq z_i\mid\mathcal F_t)=\int_{-z_i}^{\infty}\cdots\int_{-z_1}^{\infty}\int_{-\infty}^{z_0-\sum_{k=1}^iv_k}\\
&P_{X_t}(1-t^n_i-t,u)\prod_{k=1}^iP_{u+\sum_{j=k+1}^iv_j}(t^n_k-t^n_{k-1},u+\sum_{l=k}^iv_l)dudv_1\ldots dv_i
\end{align*}
Since the transition density $\pi(t,x,y)=P_x(t,y)$ is twice continuously differentiable in $x$ by assumption, it is straightforward to check that 
$$
p^{i,n}_t(z_0,\ldots,z_i)=\prod_{k=1}^i\pi(t^n_k-t^n_{k-1}, \sum_{j=0}^kz_j, \sum_{j=0}^{k-1}z_j) \pi(1-t^n_i-t, X_t, \sum_{j=0}^iz_j)
$$
One then readily obtains
$$
d\langle p^{i,n}_{.}(z_0, \ldots, z_i), B_{.}\rangle _s = \frac{1}{\pi}\frac{\partial \pi}{\partial x}(1-t^n_i-s, X_s, \sum_{j=0}^iz_k)p^{i,n}_s(z_0, \ldots, z_i)ds
$$
Hence by taking the local martingale $M$ in (\ref{eq:Ain}) to be $B$, we get
$$
A^{i,n}_t=\int_0^t\frac{1}{\pi}\frac{\partial \pi}{\partial x}(1-t^n_i-s, X_s, X_{1-t^n_i})ds
$$
Now equation (\ref{eq:An}) becomes
\begin{align*}
A^{(n)}_t=\sum_{i=0}^n1_{\{t^n_i\leq t<t^n_{i+1}\}}\Big(&\sum_{k=0}^{i-1}\int_{t^n_k}^{t^n_{k+1}}\frac{1}{\pi}\frac{\partial \pi}{\partial x}(1-t^n_k-s, X_s, X_{1-t^n_k})ds\\
&+\int_{t^n_i}^t\frac{1}{\pi}\frac{\partial \pi}{\partial x}(1-t^n_i-s, X_s, X_{1-t^n_i})ds\Big)
\end{align*}
In order to apply Theorem \ref{T:semimgX}, it only remains to prove that $E(\int_0^{\frac{1}{2}}|dA^{(n)}_s|)\leq K$ for some constant $K$ independent from $n$. The finite constant $K=\int_0^1\phi(s)ds$ works since
\begin{align*}
E\big(\int_0^{\frac{1}{2}}|dA^{(n)}_s|\big)&\leq\sum_{i=0}^n\int_{t^n_k}^{t^n_{k+1}}E\Big|\frac{1}{\pi}\frac{\partial \pi}{\partial x}(1-t^n_k-s, X_s, X_{1-t^n_k})\Big|ds\\
&\leq \sum_{i=0}^n\int_{t^n_k}^{t^n_{k+1}}\phi(1-t^n_k-s)ds = \sum_{i=0}^n\int_{1-t^n_k-t^n_{k+1}}^{1-2t^n_{k}}\phi(s)ds\\
&\leq \sum_{i=0}^n\int_{1-2t^n_{k+1}}^{1-2t^n_{k}}\phi(s)ds = \int_0^1\phi(s)ds
\end{align*}
This proves again that $B$ is a $\mathbb G$~semimartingale. Now $A^{(n)}$ converges in probability to the process $A$ given by
$$
A_t=\int_0^t\frac{1}{\pi}\frac{\partial \pi}{\partial x}(1-2s, X_s, X_{1-s})ds
$$
Since all $\mathbb G$~martingales are continuous, the comment following Theorem \ref{T:semimgX} ensures that B-A is a $\mathbb G$~martingale. Its quadratic variation is $t$, therefore it is a $\mathbb G$~Brownian motion.
\end{proof}

In the Brownian case, the result in Theorem \ref{T:diff} can also be obtained using the usual theory of initial expansion of filtration. Assume $b=0$ and $\sigma=1$, i.e.~$Z=B_{1-\cdot}$ and $X=B$.

\begin{theorem}\label{T:Brownian}
The process $B$ is a $\mathbb G$~semimartingale and
$$
B_t - \int_0^t\frac{B_{1-s}-B_s}{1-2 s}ds, \qquad 0\leq t<\frac{1}{2}
$$
is a $\mathbb G$~Brownian motion.
\end{theorem}

\begin{proof}
Introduce the filtration $\mathbb H^1=(\mathcal H_t)_{0\leq t<\frac{1}{2}}$ obtained by initially expanding $\mathbb F$ with~$B_{\frac{1}{2}}$. 
$$
\mathcal H^1_t = \bigcap_{u>t} F_u \vee \sigma(B_{\frac{1}{2}})
$$
We know that $B$ remains an $\mathbb H^1$~semimartingale and 
$$
M_t:=B_t-\int_0^t\frac{B_{\frac{1}{2}}-B_s}{\frac{1}{2}-s}ds, \qquad 0\leq t<\frac{1}{2}
$$
is an $\mathbb H^1$~Brownian motion. Now expand initially $\mathbb H^1$ with the independent $\sigma$-field $\sigma (B_v - B_{\frac{1}{2}}, \frac{1}{2}<v\leq 1)$ to obtain $\mathbb H$ i.e.
$$
\mathcal H_t=\bigcap_{u>t} H^1_u\vee \sigma(B_v - B_{\frac{1}{2}}, \frac{1}{2}<v\leq 1)
$$
Obviously $(M_t)_{0\leq t<\frac{1}{2}}$ remains an $\mathbb H$~Brownian motion. But $\mathcal G_t \subset \mathcal H_t$, for all $0\leq t<\frac{1}{2}$, hence the optional projection of $M$ onto $\mathbb G$, denoted $^oM$ in the sequel, is again a martingale (see \cite{ProtterFollmer}), i.e.
$$
^oM_t=B_t-E(\int_0^t\frac{B_{\frac{1}{2}}-B_s}{\frac{1}{2}-s}ds\mid\mathcal G_t), \qquad 0\leq t<\frac{1}{2}
$$
is a $\mathbb G$~martingale. Also, $N_t:=E(\int_0^t\frac{B_{\frac{1}{2}}-B_s}{\frac{1}{2}-s}ds\mid\mathcal G_t)-\int_0^tE(\frac{B_{\frac{1}{2}}-B_s}{\frac{1}{2}-s}\mid\mathcal G_s)ds$ is a $\mathbb G$~local martingale, see for example \cite{KLP} for a proof. So
$$
B_t\ =\ ^oM_t+N_t+\int_0^tE(\frac{B_{\frac{1}{2}}-B_s}{\frac{1}{2}-s}\mid\mathcal G_s)ds
$$
We prove now the theorem using properties of the Brownian bridge. Recall that for any $0\leq T_0 <T_1<\infty$,
\begin{equation}\label{eq:condLaw}
\mathcal L\Big( (B_t)_{T_0\leq t\leq T_1} \mid B_s, s\notin ]T_0, T_1[\Big)=\mathcal L\Big( (B_t)_{T_0\leq t\leq T_1} \mid B_{T_0}, B_{T_1} \Big)
\end{equation}
and
$$
\mathcal L\Big( (B_t)_{T_0\leq t\leq T_1} \mid B_{T_0}=x, B_{T_1}=y \Big)=\mathcal L\Big(x+\frac{t-T_0}{T_1-T_0}(y-x)+(Y^{W,T_1-T_0}_{t-T_0})_{T_0\leq t\leq T_1}\Big)
$$
where $W$ is a generic standard Brownian motion and $Y^{W, T_1-T_0}$ is the standard Brownian bridge on $[0, T_1-T_0]$. It follows that for all $ T_0\leq t\leq T_1$ and all $x$ and $y$,
\begin{equation}\label{eq:bridge}
E(B_t \mid B_{T_0} = x, B_{T_1} = y)=\frac{T_1-t}{T_1-T_0}x+\frac{t-T_0}{T_1-T_0}y
\end{equation}
For any $0\leq s< t <\frac{1}{2}$, if follows from (\ref{eq:condLaw}) and (\ref{eq:bridge}) that
$$
E(B_{\frac{1}{2}}-B_s\mid\mathcal G_s)=\frac{1}{2}(B_{1-s}-B_s)
$$
Therefore
$$
B_t-\int_0^t\frac{B_{1-s}-B_s}{1-2s}ds\ =\ ^oM_t+N_t
$$
is a $\mathbb G$~local martingale. Since the quadratic variation of the $\mathbb G$~local martingale $B-A$ is~$t$, Levy's characterization of Brownian motion ends the proof.
\end{proof}

In the immediately previous proof, the properties of the Brownian bridge allow us to compute explicitly the decomposition of $B$ in $\mathbb G$. Our method obtains both the semimartingale property and the decomposition simultaneously and generalizes to diffusions, for which the computations as in the proof of Theorem \ref{T:Brownian} are hard. We provide a shorter proof for Theorem \ref{T:Brownian} based on Theorem \ref{T:diff}.  This illustrated that, given Theorem~\ref{T:diff}, even in the Brownian case our method is shorter, simpler, and more intuitive.

\begin{proof}\textbf{[Second proof of Theorem \ref{T:Brownian}]} In the Brownian case, $\pi(t,x,y)=\frac{1}{\sqrt{2\pi t}}e^{-\frac{(y-x)^2}{2t}}$. Therefore $\frac{1}{\pi}\frac{\partial \pi}{\partial x}(t,x,y)=\frac{y-x}{t}$. Hence
$$
E\Big(\big|\frac{1}{\pi}\frac{\partial \pi}{\partial x}\big|(t-s, B_s, B_t)\Big)\leq \frac{1}{t-s}E(|B_t-B_s|)=\sqrt{\frac{2}{\pi}}\frac{1}{\sqrt{t-s}}
$$
and $\phi(x)=\sqrt{\frac{2}{\pi}}\frac{1}{\sqrt{x}}$ is integrable in zero. From the closed formula for the transition density, $A_t=\int_0^t\frac{B_{1-s}-B_s}{1-2 s}ds$. Therefore $B$ is a $\mathbb G$~semimartingale, and $B-A$ is a $\mathbb G$~Brownian motion by Theorem~\ref{T:diff}.
\end{proof}

This property satisfied by Brownian motion is inherited by diffusions whose parameters $b$ and $\sigma$ satisfy some boundedness assumptions. We add the extra assumptions that $b$ and $\sigma$ are bounded and $k\leq\sigma(x)$ for some $k>0$. The following holds.

\begin{corollary}
The process $(B_t)_{0\leq t<\frac{1}{2}}$ is a $\mathbb G$ semimartingale and 
$$
B_t - \int_0^t\frac{1}{\pi}\frac{\partial \pi}{\partial x}(1-2s, X_s, X_{1-s})ds
$$
is a $\mathbb G$~Brownian motion.
\end{corollary}

\begin{proof}
Introduce the following quantities
$$
s(x)=\int_0^x\frac{1}{\sigma(y)}dy \qquad g=s^{-1} \qquad \mu=\frac{b}{\sigma}\circ g - \frac{1}{2}\sigma^{'}\circ g
$$
The process $Y_t=s(X_t)$ satisfies the SDE $dY_t=\mu(Y_t)dt+dB_t$. The transition density is known in semi-closed form (see \cite{Zmirou}) and given by
$$
\pi(t,x,y)=\frac{1}{\sqrt{2\pi t}}\frac{1}{\sigma(y)}e^{-\frac{(s(y)-s(x))^2}{2t}}U_t(s(x),s(y))
$$
where $U_t(x,y)=H_t(x,y)e^{A(y)-A(x)}$, $H_t(x,y)=E(e^{-t\int_0^1h(x+z(y-x)+\sqrt{t}W_z)dz})$, $W$ is a Brownian bridge, $A$ a primitive of $\mu$ and $h=\frac{1}{2}(\mu^2+(\mu^{'})^2)$. It is then straightforward to compute the ratio
\begin{align*}
\frac{1}{\pi}\frac{\partial \pi}{\partial x}(t,x,y)&=\frac{1}{\sigma(x)}\Big(\frac{s(y)-s(x)}{t}+\frac{1}{U_t(s(x),s(y))}\frac{\partial U_t}{\partial x}(s(x),s(y))\Big)\\
&=\frac{1}{\sigma(x)}\Big(\frac{s(y)-s(x)}{t}+\frac{1}{H_t(s(x),s(y))}\frac{\partial H_t}{\partial x}(s(x),s(y))-\mu(s(x))\Big)
\end{align*}
From the boundedness assumptions of $b$ and $\sigma$ and their derivatives, there exists a constant $M$ such that $|\frac{1}{\pi}\frac{\partial \pi}{\partial x}(t,x,y)|\leq M(1+\frac{|s(y)-s(x)|}{t})$. Hence, for $0\leq s<t$
$$
E\Big|\frac{1}{\pi}\frac{\partial \pi}{\partial x}(t-s,X_s,X_t)\Big|\leq M\big(1+E\Big|\frac{s(X_t)-s(X_s)}{t-s}\Big|\big)
$$
But $s(X_t)-s(X_s)=Y_t-Y_s=\int_s^t\mu(Y_u)du+W_t-W_s$. But $\mu$ is bounded, hence
$$
E|s(Y_t)-s(Y_s)|\leq ||\mu||_{\infty}|t-s|+E|W_t-W_s|=||\mu||_{\infty}|t-s|+\sqrt{\frac{2}{\pi}}\sqrt{t-s}
$$
This proves the existence of a constant $C$ such that
$$
E\Big|\frac{1}{\pi}\frac{\partial \pi}{\partial x}(t-s,X_s,X_t)\Big|\leq C (1+\frac{1}{\sqrt{t-s}})
$$
Since $\phi(x)=C (1+\frac{1}{\sqrt{x}})$ is integrable in zero, we can apply Theorem \ref{T:diff} and conclude.
\end{proof}

\subsection{Application to stochastic volatility models}

Let $(\Omega, \mathcal H, P, \mathbb H)$ be a filtered probability space. Assume that we are given an $\mathbb H$~Brownian motion $W$ and a positive continuous $\mathbb H$ adapted process $\sigma$. Consider the following stochastic volatility model 
$$
dS_t=S_t\sigma_t dW_t
$$
and $\sigma$ is such that $(\sigma, A)$ is Markov w.r.t~its natural filtration with transition density $p_t((u,a),(v,b))$, where $A_t:=\int_0^t\sigma^2_sds$. Define $Z_t=\int_0^t\sigma_sdW_s$, so that $S_t=\mathcal E(Z)_t$ and let $\mathbb F$ be the filtration generated by $S$ and $\sigma$. Then $\mathcal F_t=\bigcap_{u>t}\sigma(Z_s, \sigma_s, s\leq u)$. Since $\langle Z, Z\rangle_t=A_t$, A is $\mathbb F$~adapted. 

We want to expand $(\mathcal F_t)_{0\leq t<\frac{T}{2}}$ progressively with the continuous process $X_t:=\int_{T-t}^T\sigma^2_sds=A_T-A_t$, $0\leq t<\frac{1}{2}$. The process $(X_t)_{0\leq t<\frac{T}{2}}$ satisfies Assumption \ref{A:J} with density
\begin{align*}
\int_{a_1=0}^{\infty}\int_{u_1=0}^{\infty}\ldots\int_{u_{n+1}=0}^{\infty}\prod_{k=1}^np_{t^n_{k+1}-t^n_k}&\big((u_k,A_t+\sum_{i=1}^ka_i),(u_{k+1},A_t+\sum_{i=1}^{k+1}a_i)\big)\\
&p_{\frac{T}{2}-t}\big((\sigma_t, A_t), u_1,a_1+A_t)\big)du_{n+1}\ldots du_1da_1
\end{align*}
along any refining subdivision $\{t^n_k\}$ of $[0,\frac{T}{2}]$. Let $M$ be a continuous $\mathbb F$~martingale such that $\sup_{s}|M_s|$ is integrable. Then $A^{i,n}$ and $A^{(n)}$ can be computed using the formulas in (\ref{eq:Ain}) and (\ref{eq:An}) and the $\mathbb G$~semimartingale property of $M$ will be guaranteed as soon as $E(\int_0^{\frac{T}{2}}|dA^{(n)}_s\vert)\leq K$ for some $K$ and all $n$.

\end{document}